\documentclass[a4paper,reqno,11pt]{amsart}
\usepackage{amsfonts,amsmath,amsthm,amssymb,stmaryrd}
\usepackage{mathrsfs}

\usepackage[left=1 in, right=1 in,top=1 in, bottom=1 in]{geometry}

\allowdisplaybreaks

\numberwithin{equation}{section}
\usepackage{color}

\newtheorem{theorem}{Theorem}[section]
\newtheorem{lemma}[theorem]{Lemma}
\newtheorem{proposition}[theorem]{Proposition}

\newtheorem{remark}[theorem]{Remark}

\theoremstyle{definition}
\newtheorem{definition}[theorem]{Definition}

\newcommand{\bR}{\mathbb R}

\newcommand{\Div}{\operatorname{div}}
\newcommand{\curl}{\operatorname{curl}}

\renewcommand{\epsilon}{\varepsilon}

\providecommand{\abs}[1]{\left\vert#1\right\vert}
\providecommand{\norm}[1]{\left\Vert#1\right\Vert}

\DeclareMathOperator{\diverge}{div}

\providecommand{\norm}[1]{\left\Vert#1\right\Vert}

\def\dt{\partial_t}

\def\pa{\partial}

\def\RRvert2{\right \vert\! \right\vert}
\def\Lvert3{\left \vert\!\left\vert\!\left\vert}
\def\Rvert3{\right \vert\!\right\vert\!\right\vert}

\def\nab{\nabla}

\def\dt{\partial_t}

\def\hal{\frac{1}{2}}

\def\ls{\lesssim}

\def\naba{\nab_{\mathcal{A}}}
\def\nabak{\nab_{\mathcal{A}^\kappa}}
\def\diva{\diverge_{\mathcal{A}}}
\def\divak{\diverge_{\mathcal{A}^\kappa}}

\def\a{\mathcal{A}}

\def\i{\mathcal{I}}
\def\j{\mathcal{J}}

\def\fj1{\mathcal{J}^{-1}}

\def\bp{\bar\partial}

\def\nak{\nabla_{\mathcal{A}^\kappa}}

\def\ak{{\mathcal{A}^\kappa}}

\def\fk{{\psi^\kappa}}

\makeindex

\title[Free-surface incompressible elastodynamics]{Well-posedness of the free boundary problem in incompressible elastodynamics under the mixed type stability condition}

\author{Xumin Gu}
\address[X. Gu]{School of Mathematics\\ Shanghai University of Finance and Economics\\
Shanghai 200433, China
\newline \indent and
\newline \indent The Institute of Mathematical Sciences\\
The Chinese University of Hong Kong\\
Shatin, NT, Hong Kong}
\email{gu.xumin@shufe.edu.cn}

\author{Fan Wang}
\address[F. Wang]{Department of Mathematics\\
Southwest Jiaotong University\\
Chengdu, 610031, P. R. China.
}
\email{wangf767@swjtu.edu.cn}

\thanks{X. Gu was supported by the National Natural Science Foundation of China (No.~11601305).}

\thanks{F. Wang was supported by the Fundamental Research Funds for the Central Universities(2682017CX066) and National Natural Science Foundation of China (grant no. 11701477).} 



\begin{document}


\begin{abstract}
We consider the free boundary problem for the incompressible elastodynamics equations. At the free boundary moving with the velocity of the fluid particles the columns of the deformation gradient are tangent to the boundary and the pressure vanishes outside the flow domain. We prove
the local existence of a unique smooth solution of the free boundary problem,
under the mixed type stability condition that some regions of the initial free boundary satisfy the Rayleigh-Taylor sign condition,
while the remaining boundary satisfy
the non-collinearity condition. In particular, we  solve an  open problem proposed by Y. Trakhinin in the paper \cite{Trak_18}.
\end{abstract}

\maketitle

\section{Introduction}

\subsection{Eulerian formulation}

In this paper we consider the local solutions to the free boundary problem for the incompressible elastodynamics equations:
\begin{equation}
  \label{mhd}
\begin{cases}
\dt u  + u \cdot \nabla u + \nabla  p = \nabla\cdot(GG^T)  &\text{in } \Omega (t),\\
\diverge u =0 &\text{in } \Omega (t),\\
\dt G + u \cdot \nabla G= \nabla u G&\text{in } \Omega (t),\\
\diverge (G^T) =0 &\text{in } \Omega (t).
\end{cases}
\end{equation}
In the equations \eqref{mhd}, $u$ is the velocity field, $G=G_{ij}$ is the deformation tensor, $G^T=G_{ji}$ denotes the transpose of the matrix $G$, and $p$ is the pressure function of the fluid which occupies the moving bounded domain $\Omega(t)$. Here
$$\diverge u=\partial_iu_i, \quad (\diverge G^T)_i=\partial_jG_{ji},\quad (\nabla u)_{ij}=\partial_j u_i, \quad
(\nabla uG)_{ij}=\partial_k u_iG_{kj}.$$

We require the following boundary conditions on the free surface $\Gamma(t):=\pa \Omega (t)$:
\begin{equation} \label{mhd1}
V(\Gamma(t)) = u\cdot n \quad\text{on } \Gamma(t),
\end{equation}
and
\begin{equation} \label{mhd2}
  p=0 ,\quad
G^T\cdot n=0 \quad\text{on }\Gamma(t).
\end{equation}
The equation \eqref{mhd1} is called the kinematic boundary condition which states that the free surface $\Gamma(t)$ moves with the velocity of the fluid, where $V(\Gamma(t))$ denote the normal velocity of $\Gamma(t)$
and $n$ denotes the outward unit normal of $\Gamma(t)$. Finally, we impose the initial condition
\begin{equation} \label{mhd3}
(u, G)=(u_0,G_0) \text{ on }\Omega(0),\text{ and }\Omega(0)=\Omega.
\end{equation}

\subsection{Lagrangian reformulation}\label{lagrangian}

We transform the Eulerian problem \eqref{mhd}--\eqref{mhd3} on the moving domain $\Omega(t)$ to be one on the fixed domain $\Omega$ by the use of Lagrangian coordinates. Let $\eta(x,t)\in\Omega(t)$ denote the ``position" of the fluid particle $x$ at time $t$, i.e.,
\begin{equation}
\begin{cases}
\partial_t\eta(x,t)=u(\eta(x,t),t),\quad t>0,
\\ \eta(x,0)=x.
\end{cases}
\end{equation}
We assume that $\eta(\cdot,t)$ is invertible and define the Lagrangian unknowns on $\Omega$:
\begin{equation}
v(x,t)=u(\eta(x,t),t),\ F(x,t)=G(\eta(x,t),t), \ {\mathcal F}(x,t)=\frac{\partial \eta(x, t)}{\partial x} \text{ and } q(x,t)=p(\eta(x,t),t).
\end{equation}
Then in Lagrangian coordinates, the problem \eqref{mhd}--\eqref{mhd3} becomes the following:
\begin{equation}\label{eq:mhdo}
\begin{cases}
\partial_t\eta =v &\text{in } \Omega,\\
\partial_tv  +\naba q =\naba\cdot(FF^T)  &\text{in } \Omega,\\
 \diva v = 0 &\text{in  }\Omega,\\
 \dt F=\naba v F&\text{in } \Omega,\\
\diva F^T =0 &\text{in  }\Omega,\\
q=0&\text{on }\Gamma:=\partial\Omega,
\\ F^T \cdot n  =0 &\text{on }\Gamma,\\
 (\eta,v, b)\mid_{t=0} =(\text{Id}, v_0, b_0).
 \end{cases}
\end{equation}
Here the matrix $\mathcal A=\mathcal A(\eta):={\mathcal F}^{-T}=(\nabla\eta)^{-T}$, the differential operators $\naba:=(\partial_1^{\a}, \partial_2^{\a}, \partial_3^{\a})$ with $\partial_i^{\a}=\a_{ij}\partial_j$ and $\diva g=\a_{ij}\partial_jg_i$. $n=\mathcal{A}N/\abs{\a N}$, where $N$ is the outward unit normal of $\Gamma$. Note that the kinematic boundary condition \eqref{mhd1} is automatically satisfied by the first equation in \eqref{eq:mhdo}.
We shall find out the conserved quantities for the system \eqref{eq:mhdo}. These quantities will help us reformulate the system in a proper way, and the reformulation will be more suitable for our construction of solutions. To begin with, we denote $J = {\rm det}(\nabla\eta)$, the Jacobian of the coordinate
transformation. Then we have $\dt J=0$, which implies $J=1$. Next, by the fourth equation of \eqref{eq:mhdo},
we obtain
\begin{equation}
\begin{split}
\partial_t({\mathcal F}^{-1}(t, x)F(x, t))&=\partial_t{\mathcal F}^{-1}F+{\mathcal F}^{-1}\partial_t F\\
&=-{\mathcal F}^{-1}\partial_t{\mathcal F}{\mathcal F}^{-1}F+{\mathcal F}^{-1}\partial_t F\\
&=-{\mathcal F}^{-1}\naba v {\mathcal F}{\mathcal F}^{-1}F+{\mathcal F}^{-1}\naba vF\\
&=-{\mathcal F}^{-1}\naba v F+{\mathcal F}^{-1}\naba v F\\
&=0,
\end{split}
\end{equation}
hence, we have
\begin{equation}\label{Fformula}
F={\mathcal F}G_0.
\end{equation}
It can be directly verified that $\diva F^T =0$ in $\Omega$ and $F^T \cdot n  =0 $ on $\Gamma$ if  $\Div G_0^T=0$ in $\Omega$ and $G_0^T\cdot N=0$ on $\Gamma$. That is, we must regard these two conditions as the restriction on the initial data.

Consequently, \eqref{Fformula} motivates us to eliminate $F$ from the system \eqref{eq:mhdo}. Indeed, we have:
\begin{equation}
\begin{split}
\Big(\nabla_\a\cdot (FF^T)\Big)_i&=\a_{j\ell}\partial_\ell ({\mathcal F}G_0)_{ik}({\mathcal F}G_0)_{jk}\\
&=\a_{j\ell}\partial_\ell ({\mathcal F}_{im}G_{0m k}){\mathcal F}_{jn}G_{0n k}\\
&=\delta_{n\ell}\partial_\ell ({\mathcal F}_{im}G_{0m k})G_{0n k}\\
&=\partial_\ell ({\mathcal F}_{im}G_{0m k})G_{0\ell k}\\
&=\partial_\ell (\partial_m\eta^iG_{0m k}G_{0\ell k}).
\end{split}
\end{equation}
Then the system \eqref{eq:mhdo} can be reformulated equivalently as a free-surface incompressible Euler system with a forcing term induced by the flow map:
\begin{equation}
\label{eq:mhd}
\begin{cases}
\partial_t\eta =v &\text{in } \Omega,\\
\partial_tv^i +\a_{ij}\partial_j q =\partial_\ell (\partial_m\eta^iG_{0m k}G_{0\ell k})=(G_0^T\cdot\nabla)^2\eta&\text{in } \Omega,\\
 \a_{ij}\partial_j v^i = 0 &\text{in  }\Omega,\\
 q=0 & \text{on  }\Gamma,\\
 (\eta,v)\mid_{t=0} =(\text{Id}, v_0).
 \end{cases}
\end{equation}
In the system \eqref{eq:mhd}, the initial deformation tensor $G_0$ can be regarded as a parameter vector that satisfies
\begin{equation}
\label{bcond}
\diverge G_0^T=0 \text{ in }\Omega, \text{ and }G_0^T\cdot N=0\text{ on }\Gamma.
\end{equation}

\begin{remark}
It is possible to derive the a priori estimates for the original formulation \eqref{eq:mhdo}, however, as one will see, it is crucial for us to work with the equivalent reformulation \eqref{eq:mhd}, which enables us to construct approximate solutions that are asymptotically consistent with the a priori estimates.
\end{remark}

\subsection{Previous works}
 Free boundary problems in fluid mechanics have important physical background and have been studied intensively in the mathematical community. There are a huge amount of mathematical works, we will not attempt to address exhaustive reference in this paper. In the following, let us mention two closely related systems. One is free boundary problem for the incompressible Euler equations:
Without the deformation gradient $G$, problem \eqref{mhd} becomes the free boundary problem for the incompressible Euler equations, this problem was studied by many authors in recent decades, see, for example\cite{AD, CL_00, CS07, DS_10, GMS1, GMS2, IP,IP2, Lannes, Lindblad05, MasRou, N, SZ, Wang_15, Wu1, Wu2, Wu3, Wu4, ZZ}, and the references therein. In those paper, the local in time existence of smooth solutions was proved under the Rayleigh-Taylor sign condition,
\begin{equation}\label{RTC}
-\frac{\partial p}{\partial n}\geq \lambda>0 \text{ on }\Gamma.
\end{equation}

The other is the free boundary problems for the ideal MHD models. In this case,  the deformation gradient $G$
is replaced by the magnetic field $B$.
It attracts many research interests, but up to now only few well-posedness theory for the nonlinear problem could be found. For the general plasma-vacuum interface model, the well-posedness of the nonlinear compressible problem was recently proved in Secchi and Trakhinin \cite{Secchi_14} by
the Nash-Moser iteration based on the previous results on the linearized problem \cite{Secchi_13, Trak_10}. The well-posedness of the linearized incompressible problem was proved by Morando, Trakhinin and Trebeschi \cite{Mo_14}, the nonlinear problem was sloved by Sun, Wang and Zhang \cite{Sun_17} very recently. Recall that \cite{Mo_14, Secchi_13,Secchi_14,  Sun_17, Trak_10} require that the magnetic fields on either side of the interface are not collinear. The non-collinearity condition appears as the requirement that the symbol associated to the interface is elliptic; this yields a gain of $1/2$ derivative regularity of the interface, which is crucial for closing the energy estimates. Under the assumptions of the Taylor sign condition of the total pressure and that the strength of the magnetic field is constant on the
free surface, Hao and Luo \cite{Hao_13} proved the a priori estimates for the MHD problem by adopting a geometrical point of view \cite{CL_00}, but the solution was not constructed. 
Recently, Gu and Wang \cite{GuW_2016} constructed local solutions to the free boundary problems for the ideal MHD under the Taylor sign condition and for axially MHD equations, Gu \cite{Gu_2017} proved local-wellposedness under a more "general" stability condition, which provided that the Rayleigh-Taylor sign condition is satisfied at all those points of the initial interface where the non-collinearity condition fails.
We also mention some works about the current-vortex sheet problem, which describes a velocity and magnet field discontinuity in two ideal MHD flows. The nonlinear stability of compressible current-vortex sheets was solved independently by Chen and Wang \cite{Chen_08} and Trakinin \cite{Trak_09} by using the Nash-Moser iteration. For incompressible current-vortex sheets, Coulombel, Morando, Secchi and Trebeschi \cite{CMST} proved an a priori estimate for the nonlinear problem under a strong stability condition, and Sun, Wang and Zhang \cite{Sun_15} solved the nonlinear stability very recently. 

While for the case of incompressible elastic fluid with constant density it was studied by Hao and Wang in \cite{Hao_16} where a priori estimates in Sobolev norms of solutions were derived through a geometrical
point of view of \cite{CL_00} under the fulfilment of the Rayleigh-Taylor sign
condition \eqref{RTC}. Very recently, Li, Wang and Zhang \cite{LWZ_2018} estabilished the local well-posedness for both two free boundary
problems in incompressible elastodynamics under a natural stability condition by using the
method developed in \cite{Sun_15}. For elastodynamics equations, the elasticity plays a stabilization role, see, for instance
\cite{Chen_17, Trak_18}.
In \cite{Trak_18}, Trakhinin
showed that the Rayleigh Taylor sign condition \eqref{RTC} is not necessary for well-posedness.
If the Rayleigh Taylor sign condition \eqref{RTC} fails, Trakhinin proved the local in time existence
under the non-collinearity condition \eqref{Nclinear}.
And in \cite{Trak_18}, Trakhinin proposed an open problem:
For the most general "stability" assumption for the initial data, if we just only require the fulfilment of the
Rayleigh-Taylor sign condition at all those point of $\Gamma$ where the non-collinearity condition fails,
is there local in time solution? In this paper, we will show that this is indeed the case. 

This paper is organized as follows: The main result is introduced in Section 2, and in Section 3 we will introduce some notations and recall some useful lemmas which are necessary for the proof of our theorem. Then in Section 4 we will introduce the approximation system of \eqref{eq:mhd} and then derive the uniform estimates for the approximate solutions. and the proof of Theorem \ref{mainthm} is presented in Section 5.

\section{Main results}

To avoid the use of local coordinate charts necessary for arbitrary geometries, for simplicity, we assume that the initial domain is given by
\begin{equation}
\label{domain}
\Omega=\mathbb{T}^2\times(0,1),
\end{equation}
where $\mathbb{T}^2$ denotes the 2-torus. This permits the use of one global Cartesian coordinate system. The boundary of $\Omega$ is then given by the horizontally flat bottom and top:
\begin{equation}
\Gamma =\mathbb{T}^2 \times (\{0\}\cup \{1\}).
\end{equation}
We denote $N$ by the outward unit normal vector of $\Gamma$:
\begin{equation}
N=e_3 \text{ when }x_3=1, \text{ and }N=-e_3\text{ when }x_3=0.
\end{equation}

Before stating our results of this paper, we may refer the readers to our notations and conveniences in Section \ref{nota sec}. 

Given the data $(v_0, G_0)\in H^4(\Omega)$ with $\Div v_0=\Div G_0^T=0$, as $G_0^T\cdot N=0$ on $\Gamma$, hence $G_{03i}=0, i=1, 2, 3$. We define the initial pressure function $q_0$ as the solution to the elliptic problem
\begin{equation}
\begin{cases}
-\Delta q_0=\partial_j v_{0i}\partial_i {v_0}_j-\partial_j G_{0ik}\partial_i G_{0jk}&\text{in }\Omega,\\
q_0=0&\text{on }\Gamma.
\end{cases}
\end{equation}
Denote $\Gamma=\Gamma_1\cup \Gamma_2$, and we assume initially
\begin{equation}\label{taylor}
-\nabla q_0\cdot N\geq \lambda>0 \text{ on }\Gamma_1,
\end{equation}
\begin{equation}\label{taylor1}
\abs{G_{01}\times G_{02}}\geq\delta>0~~\text{on}~~\Gamma_2,
\end{equation}
here $G_{01}=(G_{011}, G_{012}, G_{013}), G_{02}=(G_{021}, G_{022}, G_{023})$.

We define the higher order energy functional
\begin{equation}
\label{edef}
\mathfrak{E}(t)=\norm{v}_4^2+\norm{\eta}_4^2+\norm{G_0^T\cdot\nabla \eta }_4^2+\abs{\bar\partial^4\eta\cdot n }^2_{L^2(\Gamma_1)}.
\end{equation}
Here $\Big(G_0^T\cdot\nabla \eta\Big)_{ij}=G_{0ki}\partial_k\eta_j.$

Then the main result in this paper is stated as follows.
\begin{theorem}\label{mainthm}
Suppose that the initial data $ v_0 \in H^4(\Omega)$ with $\Div v_0=0$ and $G_0 \in H^4(\Omega)$ satisfies \eqref{bcond} and that the boundary condition \eqref{taylor}, \eqref{taylor1} holds initially. Then there exists a $T_0>0$ and a unique solution $(v, q, \eta)$ to  \eqref{eq:mhd} on the time interval $[0, T_0]$ which satisfies
\begin{equation}\label{enesti}
\sup_{t \in [0,T_0]} \mathfrak E(t) \leq P\left(\norm{v_0}_4^2+\norm{G_0}_4^2\right),
\end{equation}
where $P$ is a generic polynomial.

\begin{remark}
Our approach in this paper is inspired by the recent work of \cite{GuW_2016}, where the authors constructed local solutions to the free boundary problems for the ideal MHD under the Taylor sign condition. Notice that,
when we transform the Eulerian problem \eqref{mhd}--\eqref{mhd3} on the moving domain $\Omega(t)$ to be one on the fixed domain $\Omega$ by the use of Lagrangian coordinates, the final system \eqref{eq:mhd} is similar with the system studied in the paper
\cite{GuW_2016}: in \cite{GuW_2016} the forcing term is $(b_0\cdot\nabla)^2\eta$, $b_0$ is a vector, while in our case,  the forcing term is $(G_0^T\cdot\nabla)^2\eta$, $G_0$ is a matrix. Hence, we can obtained the similar estimates.
Our main innovation in this paper is the treatment of the non-collinearity case, we notice that if the initial data $G_0$ satisfy the
so-called non-collinearity condition, then we have estimate \eqref{G_0G}, which yields a gain of one derivative regularity, and help us to deal with the non-collinearity case---an  open problem proposed by Y. Trakhinin in the paper \cite{Trak_18}. 
\end{remark}

\begin{remark}
As stated in \cite[Remark 1.3]{LWZ_2018}, the authors claim they can also deal with the mixed type stability condition using the method developed
in \cite{LWZ_2018}. However, our approach is quite different from their framework.
\end{remark}

\end{theorem}

\subsection{Strategy of the proof}
The strategy of proving the local well-posedness for the inviscid free boundary problems consists of three main parts: the a priori estimates in certain energy functional spaces, a suitable approximate problem which is asymptotically consistent with the a priori estimates, and the construction of solutions to the approximate problem.  

In the usual derivation of the a priori tangential energy estimates of \eqref{eq:mhd} in the $H^4$ setting, one deduces
\begin{equation}
\begin{split}
&\hal\dfrac{d}{dt} \int_{\Omega}\abs{\bar\partial^4 v}^2+\abs{\bar\partial^4 ( G_0^T\cdot\nabla\eta)}^2+\int_{\Gamma}(-\nabla q\cdot N)\bar\partial^4{\eta}_j{\mathcal A}_{j3} q\bar\partial^4 v_i{\mathcal A}_{i3}
\\&\quad \approx \underbrace{\int_\Omega \bar\partial^4 \nabla\eta \bar\partial^4  \eta+ \bar\partial^4 \nabla\eta\bar\partial^4  q}_{\mathcal{R}_Q}+l.o.t.,
\end{split}
\end{equation}
where $l.o.t.$  denotes integrals consisting of lower-order terms. In this paper, we consider the mixed type stability condition on the boundary, we split the boundary into two parts $\Gamma=\Gamma_1\cup \Gamma_2$. Rayleigh-Taylor sign condition \eqref{taylor} is satisfied on $\Gamma_1$, and on $\Gamma_2$, non-collinearity condition \eqref{taylor1} is satisfied.
As a consequence, if the boundary satisfy the Rayleigh-Taylor sign condition, since $\dt \eta=v$, one has
\begin{equation}\label{taylores}
\begin{split}
\int_{\Gamma_1}(-\nabla q\cdot N)\bar\partial^4{\eta}_j{\mathcal A}_{j3}\bar\partial^4 v_i{\mathcal A}_{i3}&=\dfrac{1}{2}\dfrac{d}{dt}\int_{\Gamma_1}(-\nabla q\cdot N)\abs{\bar\partial^4 \eta_i{\mathcal{A}}_{i3}}^2
\\&\quad+\underbrace{\int_{\Gamma_1}  (-\nabla q\cdot N) \a_{j3}\bp^4   \eta_j   \a_{im}\pa_m   v_\ell \a_{\ell3} \bar\partial^4  \eta_i}_{\widetilde{\mathcal{R}}_Q}+l.o.t..
\end{split}
\end{equation}
Under the Rayleigh-Taylor condition, one gets that $\bar\partial^4 \eta\cdot n\in L^2(\Gamma_1)$ which is $1/2$ higher regular than $v$. For the incompressible Euler equations, through a careful study of the vorticity equation, it can be shown that $\curl\eta\in H^{4-1/2}(\Omega)$ which leads to $\eta\in H^{4+1/2}(\Omega)$ (and hence $q\in H^{4+1/2}(\Omega)$), and thus $\mathcal{R}_Q$ and $\widetilde{\mathcal{R}}_Q$ can be controlled, which in turn closes the a priori estimates in the energy functional of $\norm{v}_{4}^2+\norm{\eta}_{4+1/2}^2$, see \cite{CS07,DS_10}. However, for the incompressible elastodynamics  equations, we cannot prove that $\curl\eta\in H^{4-1/2}(\Omega)$ due to the presence of the force term. Indeed, in the study of the vorticity equation,  if one wishes to show $\curl\eta\in H^{4-1/2}(\Omega)$, then one needs to estimate $\norm{G_0^T\cdot\nabla\eta}_{4+1/2}^2$ which is out of control since one does not have $\bar\partial^4 (G_0^T\cdot\nabla\eta)\cdot n\in L^2(\Gamma_1)$ unless $G_{0}=0$ on $\Gamma_1$. As a consequence, we can only hope to close the a priori estimates in the energy functional $\mathfrak{E}(t)$ defined by \eqref{edef}. However, this yields a loss of derivatives in estimating $\mathcal{R}_Q$ and $\widetilde{\mathcal{R}}_Q$. Our idea to overcome this difficulty is, motivated by \cite{MasRou,Wang_15}, to use Alinac's good unknowns $
\mathcal{V} =\bar\partial^4 v- \bp^4\eta \cdot\nabla_\a v$ and $ \mathcal{Q} =\bar\partial^4 q- \bp^4\eta \cdot\nabla_\a q$, which derives a crucial cancellation observed by Alinhac \cite{Alinhac}, i.e., when considering the equations for $\mathcal{V} $ and $Q $, the terms $\mathcal{R}_Q$ and $\widetilde{\mathcal{R}}_Q$ disappear. On the boundary $\Gamma_2$, in this case, as $G_0^TN=0$ on $\Gamma$, hence we have $G_{03l}=0, l=1, 2, 3$. By trace theorem and the boundary satisfy the non-collinearity condition, we have $\abs{\bar\partial^4\eta}_{\frac{1}{2}}\lesssim\|G_0^T\cdot\nabla\eta\|_{H^4(\Omega)}$, which yields a gain of one derivative regularity, hence in this case, there is no derivative lose and we can use $(H^{-1/2}, H^{1/2})$ dual estimate. Combining all those estimates, the tangential energy estimates can be finished. Doing the divergence and curl estimates is somehow standard and combining with the tangential energy estimates, this allows us to close the a priori estimates of \eqref{eq:mhd} in $\mathfrak{E}(t)$.

After we obtaining the a priori estimates, we construct approximate system to \eqref{eq:mhd}, which is asymptotically consistent with the a priori estimates for the original system. This is highly nontrivial. Recalling that, under the Rayleigh-Taylor sign condition, the a priori estimates relies heavily on the geometric transport-type structure of the nonlinear problem, which will lost during the linearization approximation. Hence, we apply the nonlinear $\kappa$-approximation developed in \cite{GuW_2016} and we can derive $\kappa$-independent a priori estimates. What now remains in the proof of the local well-posedness of \eqref{eq:mhd} is to constructing solutions to the nonlinear $\kappa$-approximate problem \eqref{approximate}. This solvability can be obtained by the viscosity vanishing method used in \cite[Section 5.1]{GuW_2016}. Consequently, the construction of solutions to the incompressible elastodynamics equations \eqref{eq:mhd}  is completed.

\section{Preliminary}

\subsection{Notation}\label{nota sec}

Einstein's summation convention is used throughout the paper, and repeated
Latin indices $i,j,$ etc., are summed from 1 to 3, and repeated Greek indices
$\alpha,\beta$, etc., are summed from 1 to 2. We will use the Levi-Civita symbol
\begin{equation}\nonumber
\epsilon_{ij\ell}=\left\{\begin{aligned}
1 &,\,\,\text{even permutation of}\,\, \{1,2,3\},\\
-1 &,\,\,\text{odd permutation of}\,\, \{1,2,3\},\\
0 &,\,\,\text{otherwise.} \\
\end{aligned}\right.
\end{equation}

\subsection{Sobolev spaces on $\Omega$}

For integers $k\ge 0$ and a smooth, open domain $\Omega$ of $\mathbb{R}  ^3$, 
we define the Sobolev space $H^k(\Omega)$ ($H^k(\Omega; {\mathbb R}^3 )$) to
be the completion of $C^\infty(\Omega)$ ($C^\infty(\Omega; {\mathbb R}^3)$) 
in the norm
$$\|u\|_k := \left( \sum_{|a|\le k}\int_\Omega \left|   D^ a u(x)
\right|^2 dx\right)^{1/2},$$
for a multi-index $a \in {\mathbb Z} ^3_+$, with the standard convention that  $|a|=a_1 +a_2+ a _3$. 
For real numbers $s\ge 0$, the Sobolev spaces $H^s(\Omega)$ and the norms $\| \cdot \|_s$ are defined by interpolation. We will  write $H^s(\Omega)$ instead of $H^s(\Omega;{\mathbb R} ^3)$ 
for vector-valued functions.  And the Sobolev spaces $W^{m,p}$ can be defined similarly.
For functions $u\in H^k(\Gamma), k\geq 0$, we set
\begin{equation}
\abs{u}_k:=\Big(\sum_{|\alpha|\leq k}\int_\Gamma|\bar\partial^\alpha u(x)|^2dx\Big)^{1/2},
\end{equation}
for a multi-index $\alpha\in \mathbb{Z}_+^2.$ For real $s\geq 0$, the Hilbert space $H^s(\Gamma)$ and the boundary norm $\abs{\cdot}_s$ (or $\abs{\cdot}_{H^s(\Gamma)}$) is defined by interpolation. The negative-order Sobolev space $H^{-s}(\Gamma)$ are defined via duality: for real $s\geq 0, H^{-s}(\Gamma):=[H^{-s}(\Gamma)]^\prime.$
 
We will on occasion also refer to  the Banach space $W^{1, \infty }(\Omega)$ consisting of $L^ \infty(\Omega)$ functions whose
 weak derivatives are also in $L^ \infty (\Omega)$.

We use $D$ to denote the spatial derives, $\bar{\partial}$ to denote the tangential derivatives, $\Delta$ to denote the Laplacian on $\Omega$ and $\Delta_{\ast}$ to denote the Laplacian on $\Gamma$. We also use $\int_{\Omega} f$, $\int_{\Gamma} f$ as the integrals abbreviation of $\int_{\Omega}f\,dx$, $\int_{\Gamma}f\,dx_{\ast}$, with $x_\ast=(x_1,x_2)$.

We use $C$ to denote generic constants, which only depends on the domain $\Omega$ and the boundary $\Gamma$, and  use $f\ls g$ to denote $f\leq Cg$.  We use $P$ to denote a generic polynomial function of its arguments, and the polynomial coefficients are generic constants $C$.

\subsection{Product and commutator estimates}

We recall the following product and commutator estimates.
\begin{lemma}
It holds that

\noindent $(i)$ For $|\alpha|=k\geq 0$,
\begin{equation}
\label{co0}
\norm{D^{\alpha}(gh)}_0 \ls \norm{g}_{k}\norm{h}_{[\frac{k}{2}]+2}+\norm{g}_{[\frac{k}{2}]+2}\norm{h}_{k}.
\end{equation}
$(ii)$ For $|\alpha|=k\geq 1$, we define the commutator
\begin{equation}
[D^{\alpha}, g]h = D^{\alpha}(gh)-gD^{\alpha} h.
\end{equation}
Then we have
\begin{align}
\label{co1}
&\norm{[D^{\alpha}, g]h}_0\ls\norm{Dg}_{k-1}\norm{h}_{[\frac{k-1}{2}]+2}+\norm{Dg}_{[\frac{k-1}{2}]+2}\norm{h}_{k-1} ,\\
\label{co3}
&\abs{[D^{\alpha}, g]h}_0\ls\abs{Dg}_{k-1}\abs{h}_{[\frac{k-1}{2}]+\frac{3}{2}}+\abs{Dg}_{[\frac{k-1}{2}]+\frac{3}{2}}\abs{h}_{k-1} .
\end{align}
$(iii)$ For $|\alpha|=k\geq 2$, we define the symmetric commutator
\begin{equation}
\left[D^{\alpha}, g, h\right] = D^{\alpha}(gh)-D^{\alpha}g h-gD^{\alpha} h.
\end{equation}
Then we have
\begin{equation}
\label{co2}
\norm{\left[D^{\alpha}, g, h\right]}_0\ls\norm{Dg}_{k-2}\norm{Dh}_{[\frac{k-2}{2}]+2}+ \norm{Dg}_{[\frac{k-2}{2}]+2}\norm{Dh}_{k-2} .
\end{equation}
\end{lemma}
\begin{proof}
See, for instance \cite{GuW_2016}.
\end{proof}
We will also use the following lemma.
\begin{lemma}
It holds that
\begin{equation}
\label{co123}
\abs{gh}_{1/2} \ls \abs{g}_{W^{1,\infty}}\abs{h}_{1/2}.
\end{equation}
 \end{lemma}
\begin{proof}
It is direct to check that $\abs{gh}_{s} \ls \abs{g}_{W^{1,\infty}}\abs{h}_{s}$ for $s=0,1$. Then the estimate \eqref{co123} follows by the interpolation.
\end{proof}

\subsection{Hodge decomposition elliptic estimates}

Our derivation of high order energy estimates is based on the following Hodge-type elliptic estimates.
\begin{lemma}
\label{hodge}
Let $s\ge 1$, then it holds that
\begin{equation}
\norm{\omega}_s\ls \norm{\omega}_0+\norm{\curl \omega}_{s-1}+\norm{\Div \omega}_{s-1}+\abs{\bar{\partial}\omega \cdot N}_{s-3/2} .\label{hodd}
\end{equation}
\end{lemma}

\subsection{Normal trace estimates}

For our use of the above Hodge-type elliptic estimates, we also need the following normal trace estimates.
\begin{lemma}\label{normal trace}
It holds that
\begin{equation}\label{gga}
\abs{\bar{\partial}\omega \cdot N}_{-1/2}\ls \norm{\bar{\partial}\omega}_0+\norm{\Div\omega}_0
\end{equation}

\end{lemma}
\begin{proof}
The estimates are well-known and follow from the identity $-\Delta\omega=\curl \curl \omega-\nabla\Div \omega$. We refer the reader to Section 5.9 of \cite{Taylor}.
\end{proof}

\subsection{Horizontal convolution-by-layers and commutation estimates}

As \cite{CS07,DS_10}, we will use the operation of horizontal convolution-by-layers which is defined as follows. Let $0\le \rho(x_\ast)\in C_0^{\infty}(\bR^2)$ be a standard mollifier such that $\text{spt}(\rho)=\overline{B(0,1)}$ and $\int_{\bR^2} \rho \,dx_{\ast}=1$, with corresponding dilated function $\rho_{\kappa}(x_\ast)=\frac{1}{\kappa^2}\rho(\frac{x_\ast}{\kappa}), \kappa>0$. We then define
\begin{equation}\label{lambdakdef}
\Lambda_{\kappa}g(x_{\ast},x_3)=\int_{\bR^2}\rho_{\kappa}(x_{\ast}-y_{\ast})g(y_{\ast},x_3)\,dy_{\ast}.
\end{equation}
By standard properties of convolution, the following estimates hold:
\begin{align}
&\abs{\Lambda_{\kappa}h}_s\ls \abs{h}_s,\quad  s\ge 0, \label{test3}\\
&\abs{\bar\partial\Lambda_{\kappa}h}_0\ls \dfrac{1}{\kappa^{1-s}}\abs{h}_s,\quad 0\le s\le 1.\label{loss}
\end{align}

The following commutator estimates play an important role in the boundary estimates.
\begin{lemma}\label{comm11}
For $\kappa>0$, we define the commutator
\begin{equation}
\left[\Lambda_{\kappa}, h\right]g\equiv \Lambda_{\kappa}(h g)-h\Lambda_{\kappa}g.
\end{equation}
Then we have
\begin{align}
&\abs{[\Lambda_{\kappa}, h]g}_0\ls  \abs{h}_{L^\infty}|g|_0,\label{es0-0}\\
&\abs{[\Lambda_{\kappa}, h]\bar\partial g}_0\ls \abs{h}_{W^{1,\infty}}|g|_0,\label{es1-0}\\
&\abs{[\Lambda_{\kappa}, h]\bar\partial g}_{1/2}\ls \abs{h}_{W^{1,\infty}}\abs{g}_{1/2}.\label{es1-1/2}
\end{align}
\end{lemma}
\begin{proof}
See, for instance \cite{GuW_2016}.
\end{proof}

Our energy estimates require the use of the following:
\begin{lemma}
Let $H^{\frac{1}{2}}(\Gamma)^\prime$ denote the dual space of $H^{\frac{1}{2}}(\Gamma)$. There exists a positive constant $C$ such that
\begin{equation}
\abs{\bar\partial F}_{-\frac{1}{2}}:=\abs{\bar\partial F}_{H^{\frac{1}{2}}(\Gamma)^\prime}\leq C\abs{F}_{\frac{1}{2}}, \forall F\in H^{\frac{1}{2}}(\Gamma).
\end{equation}

\end{lemma}
\begin{proof}
See, for instance \cite{DS_10}.
\end{proof}

\begin{definition}
{\itshape{Non-collinearity condition:}}

Denote $F_j=(F_{1j}, F_{2j}, F_{3j})$ the vector field corresponding to the $j$th column
of the matrix $F$, we say $F$ satisfy the so-called non-collinearity condition if
 among the three vectors $F_1, F_2$, and $F_3$ there are two which are non-collinear at
each point of the initial free boundary, i. e.,
\begin{equation}\label{Nclinear}
\exists~~i, j\in {1, 2, 3}, i\neq j; \abs{F_i\times F_j}\geq \delta>0,~~\text{on}~~\Gamma,
\end{equation}
where $\delta$ is a fixed constant.
\end{definition}

\begin{lemma}\label{IG_0}
Let $G_0$ satisfy the non-collinearity condition \eqref{taylor1} and $G_0^T\cdot N=0$ on the initial boundary, then we have
\begin{equation}\label{G_0G}
\abs{\bar\partial^4\eta}_{\frac{1}{2}}\lesssim\|G_0^T\cdot\nabla\eta\|_{H^4(\Omega)}.
\end{equation}
\end{lemma}
\begin{proof}
On the boundary
\begin{equation*}
G_0^TN=0 \Longleftrightarrow G_{03l}=0, l=1, 2, 3,
\end{equation*}
hence
\begin{equation*}
G_0^T\cdot \nabla\eta=G_{01l}\partial_1\eta+G_{02l}\partial_2\eta,
\end{equation*}
by trace theorem, we have
\begin{equation}\label{trace}
\abs{G_{01l}\partial_1\eta+G_{02l}\partial_2\eta}_{H^{3.5}(\Gamma)}\lesssim\|G_0^T\cdot\nabla\eta\|_{H^4(\Omega)}, l=1, 2, 3.
\end{equation}
If non-collinear condition \eqref{taylor1} is satisfied, then the following matrix is  full rank, 
\begin{equation}       
\left(                 
  \begin{array}{ccc}   
    G_{011} & G_{012} & G_{013}\\  
    G_{021} & G_{022} & G_{023}\\  
  \end{array}
\right),                 
\end{equation}
and
\begin{equation*}
\begin{split}
\abs{G_{012}G_{023}-G_{022}G_{013}}^2+\abs{G_{013}G_{021}-G_{011}G_{023}}^2+\abs{G_{012}G_{021}-G_{022}G_{011}}^2\geq \delta^2,
\end{split}
\end{equation*}
then, one of them must be bigger than $\frac{\delta^2}{3}$, without the loss of generality, suppose
\begin{equation*}
\begin{split}
\abs{G_{012}G_{023}-G_{022}G_{013}}^2\geq \frac{\delta^2}{3}.
\end{split}
\end{equation*}
Assume
\begin{equation}\label{construct}
\begin{cases}
G_{012}\partial_1\eta+G_{022}\partial_2\eta&=f_1,\\
G_{013}\partial_1\eta+G_{023}\partial_2\eta&=f_2,\\
\end{cases}
\end{equation}
$f_1, f_2\in H^{3.5}(\Gamma).$

Then, we have
\begin{equation}       
\left(                 
  \begin{array}{cc}   
    \partial_1\eta \\  
    \partial_2\eta \\  
  \end{array}
\right)=\frac{P^\star}{P}
\left(                 
  \begin{array}{cc}   
    f_1 \\  
    f_2 \\  
  \end{array}
\right)          
\end{equation}

here
\begin{equation}       
P=G_{012}G_{023}-G_{022}G_{013}, 
P^\star=
\left(                 
  \begin{array}{ccc}   
    G_{023} & -G_{022} \\  
    -G_{013} & G_{012} \\ 
  \end{array}
\right),              
\end{equation}
hence, we have
\begin{equation*}
\abs{\partial_1^4\eta}_{\frac{1}{2}}\lesssim\abs{f_1}_{H^{3.5}}+\abs{f_2}_{H^{3.5}}\lesssim\|G_0^T\cdot\nabla\eta\|_{H^4(\Omega)},
\end{equation*}
furthermore, we can obtain
\begin{equation*}
\abs{\bar\partial^4\eta}_{\frac{1}{2}}\lesssim\|G_0^T\cdot\nabla\eta\|_{H^4(\Omega)}.
\end{equation*}
\end{proof}

\subsection{Geometric identities}
We recall some useful identities which can be checked directly.
For a nonsingular matrix $\mathcal F$, we have the following identities for differentiating its determinant $J$ and $\mathcal A=\mathcal F^{-T}$:
\begin{align}
\label{dJ}
&\partial J=\dfrac{\partial J}{\partial {\mathcal F}_{ij}}\partial {\mathcal F}_{ij} =  J{\mathcal{A}}_{ij}\partial {\mathcal F}_{ij},\\
&\partial {\mathcal{A}}_{ij}  = -{\mathcal{A}}_{i\ell}\partial {\mathcal F}_{m\ell}{\mathcal{A}}_{mj},
\label{partialF}
\end{align}
where $\partial$ can be $D$, $\bar{\partial}$ and $\partial_t$ operators. Moreover, we have the Piola identity
\begin{equation}\label{polia}
\partial_j\left(J{\mathcal{A}}_{ij}\right) =0.
\end{equation}

\section{Nonlinear $\kappa$-approximate problem}

Our goal of this section is to introduce our approximation of \eqref{eq:mhd} and then derive the uniform estimates for the approximate solutions.

\subsection{The nonlinear approximate $\kappa$-problem}

For $\kappa>0$, we consider the following sequence of approximate problems:
\begin{equation}\label{approximate}
\begin{cases}
\partial_t\eta =v+\fk &\text{in } \Omega,\\
\partial_tv  +\nabla_{\a^\kappa} q =\partial_m(\partial_n\eta^iG_{0n\ell}G_{0m\ell})  &\text{in } \Omega,\\
 \Div_{\a^\kappa} v = 0 &\text{in  }\Omega,\\
  q=0 & \text{on  }\Gamma,\\
 (\eta,v)\mid_{t=0} =(\text{Id}, v_0).
 \end{cases}
\end{equation}
Here the matrix $\mathcal A^{\kappa}=\a(\eta^\kappa)$ (and $J^{\kappa}$, etc.) with $\eta^\kappa$ the boundary smoother of $\eta$ defined as the solution to the following elliptic equation
\begin{equation}
\label{etadef}
\begin{cases}
-\Delta \eta^{\kappa}=-\Delta\eta &\text{in }\Omega,\\
\eta^{\kappa}=\Lambda_{\kappa}^2\eta &\text{on }\Gamma.
\end{cases}
\end{equation}
In the first equation of \eqref{approximate} we have introduced the modification term $\psi^{\kappa}=\psi^{\kappa}(\eta,v)$ as the solution to the following elliptic equation
\begin{equation}
\label{etaaa}
\begin{cases}
-\Delta \psi^{\kappa}=0&\text{in } \Omega,
\\  \psi^{\kappa}= \Delta_{*}^{-1} \mathbb{P}\left(\Delta_{*}\eta_{j}\a^{\kappa}_{j\alpha}\partial_{\alpha}{\Lambda_{\kappa}^2 v}-\Delta_{*}{\Lambda_{\kappa}^2\eta}_{j}\a^{\kappa}_{j\alpha}\partial_{\alpha} v\right) &\text{on }\Gamma,
\end{cases}
\end{equation}
where $\mathbb{P} f=f- \int_{\mathbb T^2}f$ and $\Delta_{*}^{-1}$ is the standard inverse of the Laplacian $\Delta_\ast$ on $\mathbb T^2$.

\begin{remark}
 Note that the modification term $\fk\rightarrow 0$ as $\kappa\rightarrow 0$. The introduction of $\fk$ is to eliminate two troublesome terms arising in the tangential energy estimates, which combinedly vanish as $\kappa\rightarrow 0$ but are out of control when $\kappa>0$.
\end{remark}

\begin{remark}
 Note that our definition of $\eta^{\kappa}$ is different from $\Lambda_{\kappa}^2\eta$ used in \cite{CS07,DS_10}, it only smooths $\eta$ on the boundary $\Gamma$ but not smooths horizontally in the interior $\Omega$. Though, it is sufficient to restore the nonlinear symmetric structure on the boundary in the tangential energy estimates since our $\eta^{\kappa}$ still equals to $\Lambda_{\kappa}^2\eta$ on $\Gamma$. On the other hand, our definition of $\eta^\kappa$ enables us to have the control of $\norm{G_0^T\cdot\nabla  \eta^{\kappa}}_4$ which arises in the curl and divergence estimates, while $\norm{G_0^T\cdot\nabla  \Lambda_{\kappa}^2\eta}_4$ is out of control due to the bad commutator term $\norm{[G_0^T\cdot\nabla, \Lambda_{\kappa}^2]\eta}_4$ since generally $G_{03i}\neq 0, i=1, 2, 3$ and the commutator estimates \eqref{es1-0} and \eqref{es1-1/2} only work in tangential directions.
\end{remark}

\subsection{$\kappa$-independent energy estimates}\label{apest}

For each $\kappa>0$, we can show that there exists a time $T_\kappa>0$ depending on the initial data and $\kappa>0$ such that there is a unique solution $(v, q, \eta)=(v(\kappa), q(\kappa), \eta(\kappa))$ to \eqref{approximate} on the time interval $[0,T_\kappa]$. For notational simplification, we will not explicitly write the dependence of the solution on $\kappa$. The purpose of this section is to derive the $\kappa$-independent estimates of the solutions to \eqref{approximate}, which enables us to consider the limit of this sequence of solutions as $\kappa\rightarrow 0$.

We take the time $T_{\kappa}>0$  sufficiently small so that for $t\in[0,T_{\kappa}]$,
\begin{align}
\label{ini2}
&-\nabla q (t)\cdot N\geq \dfrac{\lambda}{2} \text{ on }\Gamma_1,\\
\label{inin3}&\abs{J^{\kappa}(t)-1}\leq \dfrac{1}{8} \text{ and } \abs{\a_{ij}^{\kappa}(t)-\delta_{ij}}\leq \dfrac{1}{8} \text{ in }\Omega.
\end{align}
We define the high order energy functional:
\begin{equation}
\mathfrak{E}^{\kappa} =\norm{v}_4^2+\norm{\eta}_4^2+\norm{G_0^T\cdot\nabla \eta}_4^2+\abs{\bar\partial^4\Lambda_{\kappa}\eta_i\a^{\kappa}_{i3}}^2_{L^2(\Gamma_1)}.
\end{equation}
We will prove that $\mathfrak{E}^{\kappa}$ remains bounded on a time interval independent of $\kappa$, which is stated as the following theorem.

\begin{theorem} \label{th43}
There exists a time $T_1$ independent of $\kappa$ such that
\begin{equation}
\label{bound}
\sup_{[0,T_1]}\mathfrak{E}^{\kappa}(t)\leq 2M_0,
\end{equation}
where $M_0=P(\norm{v_0}_4^2+\norm{G_0}_4^2).$
\end{theorem}

\subsubsection{Preliminary estimates of $\eta^{\kappa}$ and $\psi^{\kappa}$}

We begin our estimates with the boundary smoother $\eta^{\kappa}$ defined by \eqref{etadef} and the modification term $\psi^{\kappa}$ defined by \eqref{etaaa}.

\begin{lemma}
\label{preest}
The following estimates hold:
\begin{align}
\label{tes1}\norm{\eta^{\kappa}}_4&\ls \norm{\eta}_4,\\
\label{tes2}\norm{G_0^T\cdot\nabla\eta^{\kappa}}_4^2&\leq P(\norm{\eta}_4,\norm{G_0}_4,\norm{G_0^T\cdot\nabla\eta}_4),\\
\label{tes0}\norm{\partial_t\eta^{\kappa} }_4&\leq P(\norm{\eta}_4, \norm{v}_4),\\
\label{fest1}\norm{\fk}_{4}&\leq P(\norm{\eta}_4,\norm{v}_3),\\
\label{fest22}\norm{G_0^T\cdot\nabla \fk}_4&\leq P(\norm{\eta}_4,\norm{v}_4,\norm{G_0^T\cdot\nabla\eta}_4),\\
\label{fest2}\norm{\partial_t \fk}_4&\leq P(\norm{\eta}_4,\norm{v}_4,\norm{\partial_t v}_3).
\end{align}
\end{lemma}
\begin{proof}
First, the standard elliptic regularity theory on the problem \eqref{etadef}, the trace theorem and the estimate \eqref{test3} yield
\begin{equation*}
\norm{\eta^{\kappa}}_4\ls \norm{\Delta\eta}_2+\abs{\Lambda_{\kappa}^2\eta}_{7/2}\ls\norm{\eta}_4+\abs{\eta}_{7/2}\ls \norm{\eta}_4,
\end{equation*}
which implies \eqref{tes1}. To prove \eqref{tes2}, we apply $G_0^T\cdot\nabla$ to \eqref{etadef} to find that, since $G_0^T\cdot N=0$ on $\Gamma$,
\begin{equation*}
\begin{cases}
-\Delta(G_0^T\cdot\nabla\eta^{\kappa})=-\Delta(G_0^T\cdot\nabla\eta)-[G_0^T\cdot\nabla, \Delta]\eta +[G_0^T\cdot\nabla, \Delta]\eta^{\kappa} &\text{in } \Omega,\\
G_0^T\cdot\nabla\eta^{\kappa}=G_0^T\cdot\nabla\Lambda_{\kappa}^2\eta &\text{on } \Gamma.
\end{cases}
\end{equation*}
We then have, since $H^2$ is a multiplicative algebra and by \eqref{tes1},
\begin{equation*}
\begin{split}
\norm{G_0^T\cdot\nabla\eta^{\kappa}}_4&\ls \norm{-\Delta(G_0^T\cdot\nabla\eta )-[G_0^T\cdot\nabla, \Delta]\eta +[G_0^T\cdot\nabla, \Delta]\eta^{\kappa}}_2
+\abs{G_0^T\cdot\nabla\Lambda_{\kappa}^2\eta}_{7/2}
\\&\ls \norm{G_0^T\cdot\nabla\eta}_4+\norm{G_0}_4(\norm{ \eta}_4+\norm{ \eta^{\kappa}}_4)
+\abs{\Lambda_{\kappa}^2(G_0^T\cdot\nabla\eta)}_{7/2}+\abs{[G_0^T\cdot\nabla, \Lambda_{\kappa}^2]\eta}_{7/2}
\\&\ls \norm{G_0\cdot\nabla\eta}_4+\norm{G_0}_4\norm{\eta}_4+\abs{G_0^T\cdot\nabla\eta}_{7/2}
+\abs{G_0}_{7/2}\abs{\eta}_{7/2} \\
&\ls \norm{G_0^T\cdot\nabla\eta}_4+\norm{G_0}_4\norm{\eta}_4.
\end{split}
\end{equation*}
Here we have used the estimates \eqref{es1-1/2} to estimate
\begin{equation*}
\begin{split}
 \abs{[G_0^T\cdot\nabla, \Lambda_{\kappa}^2]\eta}_{7/2}&\le \abs{[\Lambda_{\kappa}^2,G_{0\alpha i}]\pa_\alpha\eta}_{1/2}+\abs{[\Lambda_{\kappa}^2,G_{0\alpha i}]\pa_\alpha \bp^3\eta}_{1/2}
 +\abs{\left[\bp^3,[\Lambda_{\kappa}^2,G_{0\alpha i}]\pa_\alpha \right]\eta}_{1/2}
\\&\ls   \abs{G_0}_{W^{1,\infty}}\abs{\eta}_{7/2}
+\abs{G_0}_{7/2}\abs{\eta}_{7/2}\ls \abs{G_0}_{7/2}\abs{\eta}_{7/2}.
\end{split}
\end{equation*}
This proves \eqref{tes2}.

We now turn to prove \eqref{fest1}. By the boundary condition in \eqref{etaaa} and the elliptic theory, we obtain, using the identity \eqref{partialF}, the a priori assumption \eqref{inin3} and the estimates \eqref{tes1},
\begin{equation*}
\begin{split}
\abs{\fk}_{7/2}&\ls\abs{\Delta_{*}\eta_{j}\a^{\kappa}_{j\alpha}\partial_{\alpha}{\Lambda_{\kappa}^2 v}-\Delta_{*}{\Lambda_{\kappa}^2\eta}_{j}\a^{\kappa}_{j\alpha}\partial_{\alpha} v}_{3/2}
\ls\norm{\Delta_{*}\eta_{j}\a^{\kappa}_{j\alpha}\partial_{\alpha}{\Lambda_{\kappa}^2 v}-\Delta_{*}{\Lambda_{\kappa}^2\eta}_{j}\a^{\kappa}_{j\alpha}\partial_{\alpha} v}_{2}
\\&\ls\norm{ \eta}_4\norm{\a^{\kappa}}_{2}\norm{v}_{3}
\leq P(\norm{\eta}_4,\norm{v}_3).
\end{split}
\end{equation*}
This proves \eqref{fest1} by using further the elliptic theory and the trace theorem.

Finally, the estimate \eqref{tes0} can be obtained similarly as \eqref{tes1} by applying $\partial_t$ to \eqref{etadef} and then using the equation $\partial_t\eta =v+\fk$ and the estimate \eqref{fest1}.
The estimates \eqref{fest22} and \eqref{fest2} could be achieved similarly as \eqref{tes2} and \eqref{tes0} by applying $G_0^T\cdot\nabla$ and $\partial_t$ to \eqref{etaaa} and using the estimates \eqref{tes1}--\eqref{fest1}. This concludes the lemma.
\end{proof}

\subsubsection{Transport estimates of  $\eta$}

The transport estimate of $\eta$ is recorded as follows.
\begin{proposition}
For $t\in [0,T]$ with $T\le T_\kappa$, it holds that
\begin{equation}\label{etaest}
\norm{\eta(t)}_4^2\leq M_0+TP\left(\sup_{t\in[0,T]} \mathfrak{E}^{\kappa}(t)\right).
\end{equation}
\end{proposition}
\begin{proof}
It follows by using  $\partial_t\eta =v+\fk$ and the estimate \eqref{fest1}.
\end{proof}

\subsubsection{Pressure estimate}\label{pressure1}

In the estimates in the later sections, one needs to estimate the pressure $q$. For this, applying $J^{\kappa}\divak$ to the second equation in \eqref{approximate}, by the third equation and the Piola indentity \eqref{polia}, one gets
\begin{equation}\label{ell}
-\Div(E\nabla q )=G:=G^1+G_0^T\cdot\nabla G^2,
\end{equation}
where the matrix $E=J^{\kappa}(\a^{\kappa})^T\a^{\kappa}$ and the function $G$ is given by
\begin{align}
&G^1=J^{\kappa}\partial_t\a^{\kappa}_{ij}\partial_j v_i+\left[ J^{\kappa}\a^{\kappa}_{ij}\partial_j,G_0^T\cdot\nabla\right]G_0^T\cdot\nabla\eta_i,
 \\& G^2= J^{\kappa}\a^{\kappa}_{ij}\partial_j(G_0^T\cdot\nabla\eta_i).
\end{align}
Note that  by \eqref{inin3} the matrix $E$ is symmetric and positive.

We shall now prove the estimate for the pressure $q$.
\begin{proposition}\label{pressure}
The following estimate holds:
\begin{equation}
\label{press}
\norm{q }_4^2+\norm{\partial_tq }_3^2\leq P\left( \mathfrak{E}^{\kappa} \right).
\end{equation}
\end{proposition}
\begin{proof}
Multiplying \eqref{ell} by $q$ and then integrating by parts over $\Omega$, since $q=0$ and $G_0$ satisfies \eqref{bcond}, one obtains
\begin{equation}
\int_\Omega E\nabla q \cdot \nabla q = \int_\Omega G^1 q -\int_\Omega G^2  G_0^T\cdot\nabla q.
\end{equation}
Then we have
\begin{equation}
\norm{\nabla q}^2_0\ls \norm{G_0}_4\left(\norm{G^1}_{0}+\norm{G^2}_{0}\right)\norm{q}_{1}\leq \hal \norm{q}_1^2+ C\left(\norm{G^1}_{0}^2+\norm{G^2}_{0}^2\right)\norm{G_0}_4^2.
\end{equation}
By the Poincar\'e inequality $\norm{q}_1^2\ls \norm{\nabla q}_0^2$, one can get
\begin{equation} \label{eeeq}
\begin{split}
\norm{  q}^2_1 &\ls \norm{G_0}_4^2(\norm{ G^1}_{0}^2+\norm{ G^2}_{0}^2)
\\&\le P(\norm{G_0}_4^2, \norm{Dv}_0^2, \norm{D\partial_t\eta^{\kappa}}_0^2, \norm{D(G_0^T\cdot\nabla\eta)}_0^2, \norm{D\eta^{\kappa}}_0^2, \norm{J^{\kappa}\ak}_{L^{\infty}}^2)\leq P\left(\mathfrak{E}^{\kappa} \right).
\end{split}
 \end{equation}

Next, applying $\bar\partial^\ell $ with $\ell=1,2,3$ to the equation \eqref{ell} leads to
\begin{equation}
-\Div(E\nabla \bar\partial^\ell q )= \bar\partial^\ell G^1+G_0^T\cdot\nabla \bar\partial^\ell G^2+\left[\bar\partial^\ell, G_0^T\cdot\nabla \right]G^2
+\Div \left[\bar\partial^\ell, E\nabla \right]q.
\end{equation}
Then as for \eqref{eeeq}, we obtain
\begin{equation}\label{dddddddd}
\norm{ \bar\partial^\ell q}_{1}^2
\ls \norm{ \bar\partial^\ell G^1}_{0}^2+\norm{ \bar\partial^\ell G^2}_{0}^2
+ \norm{\left[\bar\partial^{\alpha},G_0^T\cdot\nabla\right]G_2}_{0}^2+\norm{\left[\bar\partial^{\alpha},E\nabla\right] q}_{0}^2.
\end{equation}
We then estimate the right hand side of \eqref{dddddddd}. By the estimates \eqref{tes1}--\eqref{tes0}, we may estimate
\begin{equation}\label{qqqq0}
\begin{split}
&\norm{\bar\partial^\ell G^1}_{0}^2+\norm{G_0^T\cdot\nabla \bar\partial^\ell G^2}_{0}^2+\norm{\left[\bar\partial^\ell, G_0^T\cdot\nabla \right]G^2}_{0}^2\\&\le P\left(\norm{\bar\partial^\ell \partial_t\eta^{\kappa}}_1,\norm{\bar\partial^\ell v}_1, \norm{\bar\partial^\ell \eta^{\kappa}}_1, \norm{\bar\partial^\ell (G_0^T\cdot\nabla\eta^{\kappa})}_1,\norm{\bar\partial^\ell (G_0^T\cdot\nabla\eta)}_1,\norm{G_0}_4,\norm{\eta^{\kappa}}_4\right)\\&\leq P\left( \mathfrak{E}^{\kappa} \right).
\end{split}
\end{equation}
For the last commutator term, by the H\"older inequality, we estimate for each $\ell=1,2,3,$
\begin{align}\label{qqqq1}
&\norm{[\bar\partial,E\nabla] q}_{0}\ls\norm{\bar\partial E}_{L^{\infty} }\norm{\nabla q}_{0}\ls\norm{  E}_{3}\norm{\nabla q}_{0},\\
\label{qqqq2}&\norm{[\bar\partial^2,E\nabla] q}_{0}\ls \norm{\bar\partial E}_{L^{\infty} }\norm{\bar\partial\nabla q}_{0}+\norm{\bar\partial^2E}_{L^3 }\norm{\nabla q}_{L^6 }\ls\norm{  E}_{3}\norm{\nabla q}_{1},\\
\label{qqqq3}&\norm{[\bar\partial^3,E\nabla] q}_{0}\ls\norm{\bar\partial E}_{L^{\infty} }\norm{\bar\partial^2\nabla q}_{0}+\norm{\bar\partial^2E}_{L^3 }\norm{\nabla q}_{L^6 }+\norm{\bar\partial^3 E}_{0}\norm{\nabla q}_{L^{\infty} }\ls\norm{  E}_{3}\norm{\nabla q}_{2}.
\end{align}
By \eqref{partialF}, \eqref{inin3} and \eqref{tes1} imply $\norm{E}_3\leq P(\norm{\eta}_4)$, plugging the estimates \eqref{qqqq0}--\eqref{qqqq3} into \eqref{dddddddd}, we obtain
\begin{equation}
\label{q1}
\norm{ \bar\partial^\ell q}_{1}^2  \leq  P\left( \mathfrak{E}^{\kappa} \right)\left(1+\norm{\nabla q}_{\ell-1}^2\right).
\end{equation}
On the other hand, the equation \eqref{ell} gives
\begin{equation}
\label{np}
-\partial_{33}q=\dfrac{1}{E_{33}}\left(G+\sum_{i+j\neq6}\partial_i(E_{ij}\partial_jq)+\partial_3E_{33}\partial_3q\right).
\end{equation}

This implies that we can estimate the normal derivatives of $q$ in terms of those $q$ terms with less normal derivatives. Hence, using the equation \eqref{np} and the estimates \eqref{eeeq} and \eqref{q1}, inductively on $\ell$, we obtain
\begin{equation}
\label{3q}
\norm{  q}^2_4\leq P\left( \mathfrak{E}^{\kappa} \right).
\end{equation}

We now estimate $\partial_t q$. Applying $\partial_t$ to the equation \eqref{ell} leads to
\begin{equation}
-\Div(E\nabla \partial_t q)=\partial_tG_1+G_0^T\cdot\nabla\partial_tG_2+\Div(\partial_tE\nabla q).
\end{equation}
By arguing similarly as for \eqref{3q}, we can obtain
\begin{equation}\label{3qt}
\norm{ \dt q}^2_3\leq P\left( \mathfrak{E}^{\kappa} \right).
\end{equation}
Here we have used the estimates \eqref{tes1}--\eqref{fest2} and noted that by using the second equation in \eqref{approximate} and the estimates \eqref{3q}:
\begin{equation}
\label{vt}
\norm{\partial_tv}_3^2=\norm{-\nabak q+(G_0^T\cdot \nabla)^2 \eta}_3^2 \leq P\left( \mathfrak{E}^{\kappa} \right).
\end{equation}
Consequently, the estimates \eqref{3q} and \eqref{3qt} give \eqref{press}.
\end{proof}

\subsubsection{Tangential energy estimates}\label{tan}

We start with the basic $L^2$ energy estimates.
\begin{proposition}\label{basic}
For $t\in [0,T]$ with $T\le T_\kappa$, it holds that
\begin{equation}\label{00estimate}
\norm{v(t)}_0^2+\norm{(G_0^T\cdot\nabla\eta)(t)}_0^2\leq M_0+TP\left(\sup_{t\in[0,T]}\mathfrak{E}^{\kappa}(t)\right).
\end{equation}
\end{proposition}
\begin{proof}
Taking the $L^2(\Omega)$ inner product of the second equation in \eqref{approximate} with $v$ yields
\begin{equation}\label{hhl1}
 \dfrac{1}{2}\dfrac{d}{dt}\int_{\Omega}\abs{v}^2 +\int_{\Omega} \nak q\cdot v -\int_{\Omega}\partial_\ell (\partial_m\eta^iG_{0m k}G_{0\ell k}) \cdot v^i =0.
\end{equation}
By the integration by parts and using the third equation and the boundary condition $q=0$ on $\Gamma$, using the pressure estimates \eqref{press}, we have
\begin{equation}
-\int_{\Omega} \nak q\cdot v=\int_{\Omega}\partial_j\a_{ij} q v_i\le\norm{D \a^\kappa}_{L^{\infty}}\norm{v}_0\norm{q}_0\le P\left(\sup_{t\in[0,T]}\mathfrak{E}^{\kappa}(t)\right).
\end{equation}
Since $G_0$ satisfies \eqref{bcond}, by the integration by parts and using $\dt \eta=v+\fk$, we obtain
\begin{equation}\label{hhl3}
\begin{split}
-\int_{\Omega}\partial_m(\partial_n\eta^iG_{0n\ell}G_{0m\ell}) \cdot v^i
&=\int_{\Omega}\partial_m\eta^iG_{0m k}G_{0\ell k}\partial_\ell v^i\\
&=\int_{\Omega}\partial_m\eta^iG_{0mk}G_{0\ell k}\partial_\ell\partial_t\eta^i\,dx-\int_{\Omega}\partial_m\eta^iG_{0mk}G_{0\ell k}\partial_\ell \psi^\kappa_i
\\&=\dfrac{1}{2}\dfrac{d}{dt}\int_{\Omega}\abs{G_0^T\cdot\nabla\eta}^2\,dx-\int_{\Omega}\partial_m\eta^iG_{0mk}G_{0\ell k}\partial_\ell \psi^\kappa_i
\end{split}
\end{equation}
Then \eqref{hhl1}--\eqref{hhl3} implies, using the estimates \eqref{fest1},
\begin{equation}
\dfrac{d}{dt}\int_{\Omega}\abs{v}^2+\abs{G_0^T\cdot\nabla\eta}^2 \leq P\left(\sup_{t\in[0,T]}\mathfrak E^{\kappa}(t)\right).
\end{equation}
Integrating directly in time of the above yields \eqref{00estimate}.
\end{proof}
In order to perform higher order tangential energy estimates, one needs to compute the equations satisfied by $(\bp^4 v, \bp^4 q, \bp^4 \eta)$, which requires to commutate $\bp^4$ with each term of $\pa^\ak_i$. It is thus useful to establish the following general expressions and estimates for commutators.
It turns out that it is a bit more convenient to consider the equivalent differential operator $\bp^2\Delta_\ast$ so that we can employ the structure of $\Delta_\ast \fk$ on $\Gamma$. For $i=1,2,3,$ we have
\begin{equation}
 \bp^2\Delta_\ast (\pa^\ak_if) =  \pa^\ak_i \bar\partial^2\Delta_* f + \bp^2\Delta_* {\a^{\kappa}_{ij}} \pa_j f+\left[\bar\partial^2\Delta_*, {\a^{\kappa}_{ij}} ,\pa_j f\right].
\end{equation}
By the identity \eqref{partialF}, we have that
\begin{equation}
\begin{split}
&\bp^2\Delta_* (\a^\kappa_{ij} )\pa_j f=-\bp\Delta_*(\a^\kappa_{i\ell}\bp\pa_\ell  \eta^{\kappa}_m  \a^\kappa_{mj})\pa_j f
\\&\quad=-\a^\kappa_{i\ell}\pa_\ell \bp^2\Delta_*\eta^{\kappa}_m\a^\kappa_{mj}\pa_j f
-\left[\bp\Delta_*, \a^\kappa_{i\ell}\a^\kappa_{mj} \right]\bp \pa_\ell  \eta^{\kappa}_m \pa_j f
\\&\quad=-\pa^\ak_i( \bp^2\Delta_*\eta^{\kappa}\cdot\nak f)+\bp^2\Delta_*\eta^{\kappa}\cdot\nak( \pa^\ak_i f)
-\left[\bp\Delta_*, \a^\kappa_{i\ell}\a^\kappa_{mj}\right]\bp\pa_\ell \eta^{\kappa}_m\pa_j f.
\end{split}
\end{equation}
It then holds that
\begin{equation}\label{commf}
 \bp^2\Delta_\ast (\pa^\ak_if) =  \pa^\ak_i\left(\bar\partial^2\Delta_* f- \bp^2\Delta_*\eta^{\kappa}\cdot\nak f\right)+ \mathcal{C}_i(f).
\end{equation}
where the commutator $\mathcal{C}_i(f)$ is given by
\begin{equation}
\mathcal{C}_i(f)=\left[\bar\partial^2\Delta_*, {\a^{\kappa}_{ij}} ,\pa_j f\right]+\bp^2\Delta_*\eta^{\kappa}\cdot\nak( \pa^\ak_i f)
-\left[\bp\Delta_*, \a^\kappa_{i\ell}\a^\kappa_{mj}\right]\bp\pa_\ell \eta^{\kappa}_m\pa_j f
\end{equation}
It was first observed by Alinhac \cite{Alinhac} that the highest order term of $\eta$ will be cancelled when one uses the good unknown $\bar\partial^2\Delta_* f- \bp^2\Delta_*\eta^{\kappa}\cdot\nak f$, which allows one to perform high order energy estimates.

The following lemma deals with the estimates of the commutator $\mathcal C_i(f)$.
\begin{lemma}
The following estimate holds:
\begin{equation}\label{comest}
\norm{\mathcal C_i (f)}_0\leq  P(\norm{\eta}_4) \norm{f}_4.
\end{equation}
\end{lemma}
\begin{proof}
First, by the commutator estimates \eqref{co2} and the estimates \eqref{tes1}, we have
\begin{equation}\label{Calpha1}
\norm{[\bp^2\Delta_*, \a^{\kappa}_{ij}, \partial_j f]}_0\ls\norm{\a^{\kappa}}_{3}\norm{D f}_{3}\leq P(\norm{\eta}_4) \norm{f}_4.
\end{equation}
Next, by the estimates \eqref{tes1} again, we get
\begin{equation}
\norm{\bp^2\Delta_*\eta^{\kappa}\cdot\nak( \pa^\ak_i f)}_0\le \norm{\bp^4\eta^\kappa}_0\norm{\nak( \pa^\ak_i f)}_{L^{\infty} } \leq P(\norm{\eta}_4) \norm{f}_4.
\end{equation}
Finally, by the commutator estimates \eqref{co1} and the estimates \eqref{tes1}, we obtain
\begin{equation}\label{Calpha3}
\norm{\left[\bp\Delta_*, \a^\kappa_{i\ell}\a^\kappa_{mj}\right]\bp\pa_\ell \eta^{\kappa}_m\pa_j f}_0\le \norm{\left[\bp\Delta_*, \a^\kappa_{i\ell}\a^\kappa_{mj}\right]\bp\pa_\ell \eta^{\kappa}_m}_0\norm{Df}_{L^{\infty} }
 \leq P(\norm{\eta}_4) \norm{f}_3.
\end{equation}

Consequently, the estimate \eqref{comest} follows by collecting \eqref{Calpha1}--\eqref{Calpha3}.
\end{proof}

We now introduce the good unknowns
\begin{equation}
\label{gun}
\mathcal{V}=\bar\partial^2\Delta_* v- \bp^2\Delta_*\eta^{\kappa}\cdot\nak v,\quad \mathcal{Q}=\bar\partial^2\Delta_* q- \bp^2\Delta_*\eta^{\kappa}\cdot\nak q.
\end{equation}
Applying $\bar\partial^2\Delta_*$ to the second and third equations in \eqref{approximate}, by \eqref{commf}, one gets
\begin{equation}\label{eqValpha}
\begin{split}
& \mathcal{V}_t + \nak \mathcal{Q}-(G_0^T\cdot\nabla)\left(\bar\partial^2\Delta_*(G_0^T\cdot\nabla\eta)\right)\\&\quad =F:=  \dt\left(\bp^2\Delta_*\eta^{\kappa}\cdot\nak v\right) - \mathcal{C}_i(q) +\left[\bar\partial^2\Delta_*, G_0^T\cdot\nabla\right]G_0^T\cdot\nabla\eta \text{ in }\Omega,
\end{split}
\end{equation}
and
\begin{equation} \label{divValpha}
  \nabla_\ak\cdot \mathcal{V}=- \mathcal{C}_i(v_i) \text{ in }\Omega.
\end{equation}
Note that $q=0$ on $\Gamma$ implies,
\begin{equation}\label{qValpha}
\mathcal{Q}=- \bp^2\Delta_*\Lambda_{\kappa}^2\eta_i\a^\kappa_{i3}\pa_3 q \text{ on }\Gamma.
\end{equation}

We shall now derive the $\bp^4$-energy estimates.
\begin{proposition}\label{tane}
For $t\in [0,T]$ with $T\le T_\kappa$, it holds that
\begin{equation}
\label{teee}
\norm{\bar\partial^4 v(t)}_0^2+\norm{\bar\partial^4(G_0^T\cdot\nabla\eta)(t)}_0^2+\abs{\bar\partial^4 \Lambda_{\kappa}\eta_i \a^{\kappa}_{i3}(t)}_0^2\leq M_0+TP\left(\sup_{t\in[0,T]}\mathfrak{E}^{\kappa}(t)\right).
\end{equation}
\end{proposition}
\begin{proof}
Taking the $L^2(\Omega)$ inner product of \eqref{eqValpha} with $\mathcal{V}$ yields
\begin{equation}\label{ttt1}
 \dfrac{1}{2}\dfrac{d}{dt}\int_{\Omega}\abs{\mathcal{V}}^2+\int_{\Omega} \nak \mathcal{Q}\cdot \mathcal{V} +\int_{\Omega} \bar\partial^2\Delta_*(G_0^T\cdot\nabla\eta_i)   G_0^T\cdot\nabla \mathcal{V}_i =\int_{\Omega} F\cdot \mathcal{V}.
\end{equation}
By the estimates \eqref{tes1}, \eqref{tes0}, \eqref{comest}, \eqref{vt} and  \eqref{press}, by the definition of $\mathcal{V}$, we have
\begin{equation}
\label{ggg0}
\begin{split}
\int_{\Omega} F\cdot \mathcal{V}&\le\left(\norm{\dt(\bp^4\eta^{\kappa}\cdot\nabla_{\a^{\kappa}}v) }_0+\norm{\mathcal C_i (q)}_0+\norm{[\bp^2\Delta_*, G_0^T\cdot\nabla]G_0^T\cdot\nabla\eta}_0  \right)\norm{\mathcal{V}}_0
\\&\leq \left( P(\norm{\eta}_4, \norm{v}_4, \norm{\dt v}_3)+P(\norm{\eta}_4)\norm{q}_4+\norm{G_0}_4\norm{G_0^T\cdot\nabla\eta}_4\right)P(\norm{\eta^\kappa}_4, \norm{v}_4)
\\   &\leq P\left(\sup_{t\in[0,T]} \mathfrak{E}^{\kappa}(t)\right).
\end{split}
\end{equation}
We now estimate the second and third terms on the left hand side of \eqref{ttt1}. By the definition of $\mathcal V$ and  recalling that $v=\partial_t\eta-\fk$, we have that for the third term
\begin{align}
\label{gg1}
&\int_{\Omega}  \bar\partial^2\Delta_*(G_0^T\cdot\nabla\eta_i)    G_0^T\cdot\nabla \mathcal{V}_i\nonumber
\\&\quad= \int_{\Omega}  \bar\partial^2\Delta_*(G_0^T\cdot\nabla\eta_i)   G_0^T\cdot\nabla\left(\bar\partial^2\Delta_* v_i- \bp^2\Delta_*\eta^{\kappa}\cdot\nak v_i\right)\nonumber
\\&\quad=\int_{\Omega}  \bar\partial^2\Delta_*(G_0^T\cdot\nabla\eta_i)   \left(\bar\partial^2\Delta_*(G_0^T\cdot\nabla v_i)-[\bar\partial^2\Delta_*,G_0^T\cdot\nabla] v_i-G_0^T\cdot\nabla(\bp^2\Delta_*\eta^{\kappa}\cdot\nak v_i)\right)\nonumber
\\&\quad=\hal \dfrac{d}{dt}\int_{\Omega}\abs{\bar\partial^2\Delta_*(G_0^T\cdot\nabla\eta)}^2\nonumber
\\&\quad\quad-\int_{\Omega}  \bar\partial^2\Delta_*(G_0^T\cdot\nabla\eta_i)   \left(\bar\partial^2\Delta_*(G_0^T\cdot\nabla  \psi^\kappa_i)+[\bar\partial^2\Delta_*,G_0^T\cdot\nabla] v_i+G_0^T\cdot\nabla(\bp^2\Delta_*\eta^{\kappa}\cdot\nak v_i)\right)\nonumber
\\ &\quad\ge \hal \dfrac{d}{dt}\int_{\Omega}\abs{\bar\partial^2\Delta_*(G_0^T\cdot\nabla\eta)}^2  - \norm{G_0^T\cdot\nabla\eta}_4 P(\norm{G_0^T\cdot\nabla\fk}_4,\norm{G_0}_4,\norm{v}_4,\norm{\eta^\kappa}_4,\norm{G_0^T\cdot\nabla\eta^\kappa}_4)\nonumber
\\ &\quad\ge \hal \dfrac{d}{dt}\int_{\Omega}\abs{\bar\partial^2\Delta_*(G_0^T\cdot\nabla\eta)}^2  - P\left(\sup_{t\in[0,T]} \mathfrak{E}^{\kappa}(t)\right) .
\end{align}
Here we have used the commutator estimates \eqref{co1}, the estimates \eqref{tes1}, \eqref{tes2} and \eqref{fest22}.

Now for the second term, by the integration by parts and using the equation \eqref{divValpha} and the boundary condition \eqref{qValpha}, we obtain
\begin{equation}
\label{gg2}
\begin{split}
\int_{\Omega} \nak \mathcal{Q}\cdot \mathcal{V}&=\int_{\Gamma} \mathcal{Q}\a^{\kappa}_{i\ell}N_\ell\mathcal{V}_i -\int_{\Omega} \mathcal{Q} \nabla_\ak\cdot \mathcal{V}-\int_{\Omega}\pa_\ell (\a^{\kappa}_{i\ell} )\mathcal{Q} \mathcal{V}_i\\&=
- \int_{\Gamma_1} \pa_3 q \bp^2\Delta_*(\Lambda_\kappa^2\eta_j)\a^\kappa_{j3} \a^{\kappa}_{i\ell}N_\ell\mathcal{V}_i - \int_{\Gamma_2} \pa_3 q \bp^2\Delta_*(\Lambda_\kappa^2\eta_j)\a^\kappa_{j3} \a^{\kappa}_{i\ell}N_\ell\mathcal{V}_i \\
&\ \ \ \ +\underbrace{\int_{\Omega} \mathcal{Q} \mathcal{C}_i(v_i)-\pa_\ell (\a^{\kappa}_{i\ell} )\mathcal{Q} \mathcal{V}_i }_{\mathcal R}.
\end{split}
\end{equation}
By the definition of $\mathcal{Q}$ and $\mathcal{V}$ and the estimates \eqref{press}, \eqref{comest} and \eqref{tes1},  we have
\begin{equation}
\label{gggg3}
\begin{split}
  \mathcal R&\ls\norm{\mathcal Q}_0(\norm{\mathcal C_i(v_i)}_0+\norm{D\a^{\kappa}}_{L^{\infty} }\norm{\mathcal V}_0)
 \\&\le P(\norm{\eta^\kappa}_4, \norm{q}_4)\left(P(\norm{\eta}_4)\norm{v}_4
 +\norm{D\a^{\kappa}}_{L^{\infty} }P(\norm{\eta^\kappa}_4, \norm{v}_4)\right)
 \\&\leq  P\left(\sup_{t\in[0,T]} \mathfrak{E}^{\kappa}(t)\right).
 \end{split}
\end{equation}
\noindent {\itshape{Estimates $\Gamma_1$ under the fulfilment of the Rayleigh-Taylor sign condition}}

By the definition of $\mathcal{V}$, since $\eta^{\kappa}=\Lambda_\kappa^2\eta$ on $\Gamma$ and $v=\partial_t\eta-\fk$, we have
\begin{equation}
\begin{split}
&- \int_{\Gamma_1} \pa_3 q\bp^2\Delta_* \Lambda_\kappa^2\eta_j \a^\kappa_{j3} \a^{\kappa}_{i\ell}N_\ell\mathcal{V}_i
\\&\quad = \int_{\Gamma_1} (-\nabla q\cdot N )\a^\kappa_{j3}\bp^2\Delta_* \Lambda_\kappa^2\eta_j  \a^{\kappa}_{i3} \mathcal{V}_i
 \\&\quad=\int_{\Gamma_1}(-\nabla q\cdot N )\a^\kappa_{j3}\bp^2\Delta_* \Lambda_\kappa^2\eta_j  \a^{\kappa}_{i3} (\bar\partial^2\Delta_* v_i- \bp^2\Delta_* \Lambda_\kappa^2\eta \cdot\nak v_i)
  \\&\quad=\int_{\Gamma_1}(-\nabla q\cdot N )\a^\kappa_{j3}\bp^2\Delta_* \Lambda_\kappa^2\eta_j  \a^{\kappa}_{i3} \left(\bar\partial^2\Delta_* \partial_t\eta_i - \bar\partial^2\Delta_* \psi^\kappa _i - \bp^2\Delta_* \Lambda_\kappa^2\eta \cdot\nak v_i\right)  .
\end{split}
\end{equation}
Note that
\begin{equation}
\begin{split}
&\int_{\Gamma_1}(-\nabla q\cdot N )\a^\kappa_{j3}\bp^2\Delta_* \Lambda_\kappa^2\eta_j  \a^{\kappa}_{i3} \bar\partial^2\Delta_* \partial_t\eta_i  \\
  &\quad
=\int_{\Gamma_1}(-\nabla q\cdot N )\a^\kappa_{j3}\bp^2\Delta_* \Lambda_\kappa \eta_j  \a^{\kappa}_{i3} \bar\partial^2\Delta_* \Lambda_\kappa \partial_t\eta_i
 + \int_{\Gamma_1}\bp^2\Delta_* \Lambda_\kappa \eta_j \left[\Lambda_\kappa,  -\nabla q\cdot N \a^\kappa_{j3}\a^{\kappa}_{i3}\right] \bar\partial^2\Delta_* \partial_t\eta_i \\
&\quad
=\hal \frac{d}{dt}\int_{\Gamma_1}(-\nabla q\cdot N )\abs{\a^\kappa_{i3}\bp^2\Delta_* \Lambda_\kappa \eta_i  }^2  +\hal \int_{\Gamma_1}\dt( \nabla q\cdot N ) \abs{\a^\kappa_{i3}\bp^2\Delta_* \Lambda_\kappa \eta_i  }^2
\\&\qquad +\int_{\Gamma_1} \nabla q\cdot N  \a^\kappa_{j3}\bp^2\Delta_* \Lambda_\kappa \eta_j  \dt\a^{\kappa}_{i3} \bar\partial^2\Delta_*  \Lambda_\kappa \eta_i
  +\int_{\Gamma_1}\bp^2\Delta_* \Lambda_\kappa \eta_j \left[\Lambda_\kappa, -\nabla q\cdot N \a^\kappa_{j3}\a^{\kappa}_{i3}\right] \bar\partial^2\Delta_* \partial_t\eta_i .
\end{split}
\end{equation}
Therefore, we obtain
\begin{equation}
\label{gg3}
\begin{split}
&- \int_{\Gamma_1} \pa_3 q\bp^2\Delta_* \Lambda_\kappa^2\eta_j \a^\kappa_{j3} \a^{\kappa}_{i\ell}N_\ell\mathcal{V}_i
 \\
&\quad
=\hal \frac{d}{dt}\int_{\Gamma_1}(-\nabla q\cdot N )\abs{\a^\kappa_{i3}\bp^2\Delta_* \Lambda_\kappa \eta_i  }^2
+\underbrace{\hal \int_{\Gamma_1}\dt(\nabla q\cdot N )\abs{\a^\kappa_{i3}\bp^2\Delta_* \Lambda_\kappa \eta_i  }^2   }_{\mathcal{I}_1}
\\&\qquad +\underbrace{\int_{\Gamma_1} \nabla q\cdot N  \a^\kappa_{j3}\bp^2\Delta_* \Lambda_\kappa \eta_j  \dt\a^{\kappa}_{i3} \bar\partial^2\Delta_*  \Lambda_\kappa \eta_i   }_{\mathcal{I}_2}
+\underbrace{\int_{\Gamma_1}\bp^2\Delta_* \Lambda_\kappa \eta_j \left[\Lambda_\kappa,  -\nabla q\cdot N \a^\kappa_{j3}\a^{\kappa}_{i3}\right] \bar\partial^2\Delta_* \partial_t\eta_i }_{\mathcal{I}_3}
\\&\qquad +\underbrace{\int_{\Gamma_1} \nabla q\cdot N \a^\kappa_{j3}\bp^2\Delta_* \Lambda_\kappa^2\eta_j  \a^{\kappa}_{i3}  \bp^2\Delta_* \Lambda_\kappa^2\eta \cdot\nak v_i   }_{\mathcal{I}_4}
  +\underbrace{\int_{\Gamma_1} \nabla q\cdot N \a^\kappa_{j3}\bp^2\Delta_* \Lambda_\kappa^2\eta_j  \a^{\kappa}_{i3}  \bar\partial^2\Delta_* \psi^\kappa_i  }_{\mathcal{I}_5}.
\end{split}
\end{equation}

We now estimate $\mathcal{I}_1$--$\i_5$. By the estimates \eqref{press}, we deduce
\begin{equation}
\label{i0}
\mathcal{I}_1 \ls \abs{\partial_3\partial_t q}_{L^{\infty}}\abs{\a^\kappa_{i3}\bp^2\Delta_* \Lambda_\kappa \eta_i  }^2\ls \norm{\partial_tq}_3\abs{\a^\kappa_{i3}\bp^2\Delta_* \Lambda_\kappa \eta_i  }^2\le P\left(\sup_{t\in[0,T]} \mathfrak{E}^{\kappa}(t)\right).
\end{equation}
By the identity \eqref{partialF}, we have
\begin{align}
\nonumber \i_2&= \int_{\Gamma_1} (-\nabla q\cdot N) \a^\kappa_{j3}\bp^2\Delta_* \Lambda_\kappa \eta_j   \a^{\kappa}_{i\ell}\pa_\ell\dt\eta^{\kappa}_m\a^\kappa_{m3} \bar\partial^2\Delta_*  \Lambda_\kappa \eta_i
\\&= \underbrace{\int_{\Gamma_1}  (-\nabla q\cdot N) \a^\kappa_{j3}\bp^2\Delta_* \Lambda_\kappa \eta_j   \a^{\kappa}_{i3}\pa_3\dt\eta^{\kappa}_m\a^\kappa_{m3} \bar\partial^2\Delta_*  \Lambda_\kappa \eta_i   }_{\i_{2a}}\nonumber
\\&\quad+\int_{\Gamma_1}  (-\nabla q\cdot N)  \a^\kappa_{j3}\bp^2\Delta_* \Lambda_\kappa \eta_j   \a^{\kappa}_{i\alpha}\pa_\alpha\dt\Lambda_{\kappa}^2\eta_m\a^\kappa_{m3} \bar\partial^2\Delta_*  \Lambda_\kappa \eta_i.\label{tmp1}
\end{align}
As usual, we obtain
\begin{equation}
\label{1a}
\i_{2a} \ls\abs{\a^\kappa_{j3}\bp^2\Delta_* \Lambda_\kappa \eta_j  }_0^2\abs{\partial_3 q\pa_3\dt\eta^{\kappa}_m\a^\kappa_{m3}}_{L^{\infty}}\le P\left(\sup_{t\in[0,T]} \mathfrak{E}^{\kappa}(t)\right).
\end{equation}
On the other hand, using $\dt\eta=v+\fk$ we have
\begin{align}
&\int_{\Gamma_1} (-\nabla q\cdot N) \a^\kappa_{j3}\bp^2\Delta_* \Lambda_\kappa \eta_j   \nonumber \a^{\kappa}_{i\alpha}\pa_{\alpha}\dt\Lambda_{\kappa}^2\eta_m\a^\kappa_{m3} \bar\partial^2\Delta_*  \Lambda_\kappa \eta_i
\\&\quad=\underbrace{\int_{\Gamma_1}  (-\nabla q\cdot N) \a^\kappa_{j3}\bp^2\Delta_* \Lambda_\kappa \eta_j   \a^{\kappa}_{i\alpha}\pa_\alpha \Lambda_{\kappa}^2 v_m \a^\kappa_{m3} \bar\partial^2\Delta_*  \Lambda_\kappa \eta_i   }_{\i_{2b}}\nonumber
\\&\qquad+\underbrace{\int_{\Gamma_1}  (-\nabla q\cdot N) \a^\kappa_{j3}\bp^2\Delta_* \Lambda_\kappa \eta_j   \a^{\kappa}_{i\alpha}\pa_\alpha \Lambda_{\kappa}^2 \psi^\kappa_m\a^\kappa_{m3} \bar\partial^2\Delta_*  \Lambda_\kappa \eta_i   }_{\i_{2c}}
\label{ttt}
\end{align}
To estimate $\i_{2c}$, the difficulty is that one can not have an $\kappa$-independent control of $\abs{\bar\partial^2\Delta_*  \Lambda_\kappa \eta_i }_0$. Our observation is that since $\fk\rightarrow 0$ as $\kappa\rightarrow 0$, this motives us to deduce the following estimates:
 \begin{equation}\label{lwuqing}
 \abs{\bar\partial   \fk}_{L^{\infty} }\leq \sqrt\kappa P(\norm{\eta}_4,\norm{v}_3).
 \end{equation}
Indeed, we can rewrite the boundary condition in \eqref{etaaa} as
\begin{align}
 \fk=\Delta_{*}^{-1}\mathbb{P} f^\kappa,\ f^\kappa:= \Delta_{*} (\eta_j-{\Lambda_{\kappa}^2\eta_j})\a^{\kappa}_{j\alpha}\partial_{\alpha}{\Lambda_{\kappa}^2 v}-\Delta_{*}{\Lambda_{\kappa}^2\eta}_j\a^{\kappa}_{j\alpha}\partial_{\alpha}(v-{\Lambda_{\kappa}^2 v}) .
\end{align}
By using Morrey's inequality and the Sobolev embeddings and the trace theorem,
\begin{equation}
\abs{g-\Lambda_{\kappa}g}_{L^\infty}\ls \sqrt{\kappa}\abs{\bp g}_{L^4} \ls \sqrt{\kappa}\abs{g}_{3/2}\ls \sqrt{\kappa}\norm{g}_{2},
\end{equation}
we obtain
\begin{equation}
\begin{split}
\abs{f^\kappa}_{L^\infty}&\ls\abs{\a^{\kappa}_{j\alpha}\partial_{\alpha}\Lambda_{\kappa}^2 v}_{L^\infty}\abs{\Delta_*\eta-\Lambda_{\kappa}^2\Delta_*\eta}_{L^\infty}+ \abs{\Delta_{*}{\Lambda_{\kappa}^2\eta}_j\a^{\kappa}_{j\alpha}}_{L^\infty}\abs{\partial_{\alpha}v-\Lambda_{\kappa}^2 \partial_{\alpha}v}_{L^\infty} \\&\ls\sqrt\kappa P(\norm{\eta}_4,\norm{v}_3).
\end{split}
\end{equation}
Then by the elliptic estimate and the Sobolev embeddings, we deduce
 \begin{equation}
\abs{\bar\partial \fk}_{L^{\infty} }\ls\abs{\bar\partial \fk}_{W^{1,4}}\ls\abs{f^\kappa}_{L^4}\ls\abs{f^\kappa}_{L^\infty} \ls \sqrt\kappa P(\norm{\eta}_4,\norm{v}_3),
 \end{equation}
 which proves \eqref{lwuqing}. Hence, by \eqref{lwuqing} together with \eqref{loss}, we have
\begin{equation}\label{22c}
\begin{split}
\i_{2c}&\ls \abs{\partial_3q\a^\kappa_{m3}\a^{\kappa}_{i\alpha}}_{L^{\infty} }\abs{\a^\kappa_{j3}\bp^2\Delta_* \Lambda_\kappa \eta_j }_0\abs{\bar\partial^2\Delta_* \Lambda_\kappa \eta_i}_0\abs{\pa_\alpha \Lambda_{\kappa}^2 \psi^\kappa_m}_{L^{\infty} }
\\ &\ls\abs{\partial_3q\a^\kappa_{m3}\a^{\kappa}_{i\alpha}}_{L^{\infty} }\abs{\a^\kappa_{j3}\bp^2\Delta_* \Lambda_\kappa \eta_j }_0\dfrac{1}{\sqrt{\kappa}}\abs{\eta}_{7/2}\sqrt{\kappa}P(\norm{\eta}_4,\norm{v}_3)
\\&\leq P\left(\sup_{t\in[0,T]} \mathfrak{E}^{\kappa}(t)\right).
\end{split}
\end{equation}
Note that the term $\i_{2b}$ is out of control by an $\kappa$-independent bound alone.

For $\i_3$, by the commutator estimates \eqref{es1-1/2}, \eqref{co123}, \eqref{press}, \eqref{tes1} and \eqref{fest1}, we obtain
\begin{align}
\nonumber
\i_3&\leq \abs{\bp^2\Delta_* \Lambda_\kappa \eta_j }_{-1/2}\abs{\left[\Lambda_\kappa, (-\nabla q\cdot N )\a^\kappa_{j3}\a^{\kappa}_{i3}\right] \bar\partial(\bp\Delta_* \partial_t\eta_i)}_{1/2}
\\\nonumber &\ls \abs{\bp\Delta_* \Lambda_\kappa \eta_j }_{1/2}\abs{ \partial_3 q\a^{\kappa}_{j3}\a^{\kappa}_{i3}}_{W^{1,\infty} }\abs{\bp\Delta_*\partial_t\eta}_{1/2}\nonumber
\\&\ls \norm{\eta}_{4}\norm{\partial_3 q\a^{\kappa}_{j3}\a^{\kappa}_{i3}}_3\norm{v+\psi^{\kappa}}_{4} \leq P\left(\sup_{t\in[0,T]} \mathfrak{E}^{\kappa}(t)\right).
\label{commm1}
\end{align}

To control $\i_4$, similarly as \eqref{tmp1}, we write
\begin{align}
\nonumber
\i_4=&\underbrace{\int_{\Gamma_1}(\nabla q\cdot N )\a^\kappa_{j3}\bp^2\Delta_* \Lambda_\kappa^2\eta_j  \a^{\kappa}_{m3}  \bp^2\Delta_*\Lambda_\kappa^2\eta_{i}\a^{\kappa}_{i3}\partial_3 v_m   }_{\i_{4a}}\\&+\underbrace{\int_{\Gamma_1}(\nabla q\cdot N )\a^\kappa_{j3}\bp^2\Delta_* \Lambda_\kappa^2\eta_j \a^{\kappa}_{i\alpha} \partial_{\alpha} v_m\a^{\kappa}_{m3} \bp^2\Delta_*\Lambda_\kappa^2\eta_i  }_{\i_{4b}}.\label{ttttt}
\end{align}
By the commutator estimates \eqref{es1-0}, we have
\begin{equation}
\label{ennennene}
\begin{split}
\abs{\a^\kappa_{j3}\bp^2\Delta_* \Lambda_\kappa^2\eta_j }_0&\ls \abs{[\Lambda_{\kappa}, \a^\kappa_{j3}]\bp (\bp\Delta_*\Lambda_{\kappa}\eta_j)}_0+\abs{\a^\kappa_{j3}\bp^2\Delta_*(\Lambda_\kappa\eta_j)}_0
\\&\ls\abs{\a^{\kappa}}_{W^{1,\infty} }\abs{\bp\Delta_*\Lambda_{\kappa}\eta_j}_0+\abs{\a^\kappa_{j3}\bp^2\Delta_*(\Lambda_\kappa\eta_j)}_0.
\end{split}
\end{equation}
Then we obtain
\begin{equation}\label{3b}
\i_{4a}\ls
\left(\abs{\a^\kappa_{j3}\bp^2\Delta_* \Lambda_\kappa^2\eta_j }_0^2+\abs{\ak}_{W^{1,\infty}}^2\norm{\eta}_4^2\right)\abs{\partial_3q\a^{\kappa}_{m3}\partial_3 v_m}_{L^{\infty} }
 \leq P\left(\sup_{t\in[0,T]} \mathfrak{E}^{\kappa}(t)\right).
\end{equation}
Note that the term $\i_{4b}$ is also out of control by an $\kappa$-independent bound alone.

Now we take care of $\i_{2b}$ and $\i_{4b}$. Notice that $\mathcal{I}_{2b}$ and $\mathcal{I}_{4b}$ are cancelled out in the limit $\kappa\rightarrow0$, however, it is certainly not the case when $\kappa>0$. This is most involved thing in the tangential energy estimates. Note also that we can not use the commutator estimate to interchange the position of the mollifier operator $\Lambda_\kappa$ in each of two terms since  $\abs{\bar\partial^2\Delta_*\eta}_{L^\infty}$ is out of control.
The key point here is to use the term $\i_5$, by the definition of the modification term $\fk$, to kill out both $\i_{2b}$ and $\i_{4b}$; this is exactly the reason that we have introduced $\fk$. By the boundary condition in \eqref{etaaa}, we deduce
\begin{equation}
\begin{split}
\i_5&=\int_{\Gamma_1}\nabla q\cdot N \a^\kappa_{j3}\bp^2\Delta_* \Lambda_\kappa^2\eta_j  \a^{\kappa}_{i3}  \bar\partial^2 \left( \Delta_{*}\eta_m \a^\kappa_{m\alpha}\pa_\alpha{\Lambda_{\kappa}^2 v_i}-\Delta_{*}\Lambda_{\kappa}^2\eta_m \a^\kappa_{m\alpha}\pa_\alpha{ v_i}\right)
\\&=\int_{\Gamma_1} \nabla q\cdot N  \a^\kappa_{j3}\bp^2\Delta_* \Lambda_\kappa^2\eta_j  \a^{\kappa}_{i3}    \bar\partial^2 \Delta_{*}\eta_m \a^\kappa_{m\alpha}\pa_\alpha{\Lambda_{\kappa}^2 v_i}
\\&\quad+\underbrace{\int_{\Gamma_1} (-\nabla q\cdot N ) \a^\kappa_{j3}\bp^2\Delta_* \Lambda_\kappa^2\eta_j  \a^{\kappa}_{i3}
\bar\partial^2 \Delta_{*}\Lambda_{\kappa}^2\eta_m \a^\kappa_{m\alpha}\pa_\alpha{ v_i} }_{-\i_{4b}}
\\&\quad+\underbrace{\int_{\Gamma_1} \nabla q\cdot N \a^\kappa_{j3}\bp^2\Delta_* \Lambda_\kappa^2\eta_j  \a^{\kappa}_{i3}  \left( \left[\bar\partial^2, \a^\kappa_{m\alpha}\pa_\alpha{\Lambda_{\kappa}^2 v_i}\right]\Delta_{*}\eta_m -\left[\bar\partial^2,\a^\kappa_{m\alpha}\pa_\alpha{ v_i}\right] \Delta_{*}\Lambda_{\kappa}^2\eta_m \right)  }_{\i_{5a}}
\label{tmp2}
\end{split}
\end{equation}
By doing estimates as usual and using \eqref{ennennene} again, we have
\begin{equation}
\label{2a}
\begin{split}
\i_{5a} &\ls |\partial_3q\a^{\kappa}_{i3}|_{L^{\infty}(\Omega)}\abs{\a^\kappa_{j3}\bp^2\Delta_* \Lambda_\kappa^2\eta_j }_0\abs{\left(\left[\bar\partial^2, \a^\kappa_{m\alpha}\pa_\alpha{\Lambda_{\kappa}^2 v_i}\right]\Delta_{*}\eta_m -\left[\bar\partial^2,\a^\kappa_{m\alpha}\pa_\alpha{ v_i}\right] \Delta_{*}\Lambda_{\kappa}^2\eta_m \right)}_0
\\
&\le P\left(\sup_{t\in[0,T]} \mathfrak{E}^{\kappa}(t)\right).
\end{split}
\end{equation}
We rewrite the first term as
\begin{align}
\nonumber&\int_{\Gamma_1} \nabla q\cdot N \a^\kappa_{j3}\bp^2\Delta_* \Lambda_\kappa^2\eta_j  \a^{\kappa}_{i3}    \bar\partial^2 \Delta_{*}\eta_m \a^\kappa_{m\alpha}\pa_\alpha{\Lambda_{\kappa}^2 v_i}
\\\nonumber&\quad=\underbrace{\int_{\Gamma_1} \nabla q\cdot N  \a^\kappa_{j3}\bp^2\Delta_* \Lambda_\kappa \eta_j  \a^{\kappa}_{i3}    \bar\partial^2 \Delta_{*} \Lambda_\kappa \eta_m \a^\kappa_{m\alpha}\pa_\alpha{\Lambda_{\kappa}^2 v_i}}_{-\i_{2b}}
\\ &\qquad+\underbrace{\int_{\Gamma_1}\bp^2\Delta_* \Lambda_\kappa \eta_j \left[\Lambda_\kappa,  \nabla q\cdot N \a^\kappa_{j3}\a^{\kappa}_{i3} \a^\kappa_{m\alpha}\pa_\alpha{\Lambda_{\kappa}^2 v_i}\right]\bar\partial^2 \Delta_{*}\eta_m   }_{\i_{5b}}.\label{temp3}
\end{align}
By arguing similarly as \eqref{commm1} for $\i_3$, we have
\begin{equation}
\label{pw}
\i_{5b}\le \norm{\eta}_{4}\norm{\partial_3 q\a^\kappa_{j3}\a^{\kappa}_{i3} \a^\kappa_{m\alpha}\pa_\alpha{\Lambda_{\kappa}^2 v_i}}_3\norm{\eta}_{4}\leq P\left(\sup_{t\in[0,T]} \mathfrak{E}^{\kappa}(t)\right).
\end{equation}

Now combining \eqref{tmp1}, \eqref{ttt}, \eqref{ttttt}, \eqref{tmp2} and \eqref{temp3},  and using the estimates \eqref{1a}, \eqref{22c}, \eqref{3b}, \eqref{2a} and \eqref{pw}, we deduce
\begin{equation}\label{ggiip}
\i_2+\i_4+\i_5=\i_{2a}+\i_{2c}+\i_{4a}+\i_{5a}+\i_{5b}\le P\left(\sup_{t\in[0,T]} \mathfrak{E}^{\kappa}(t)\right).
\end{equation}

\noindent {\itshape{Estimates $\Gamma_2$ under the fulfilment of the non-collinearity condition}}

By the definition of $\mathcal{V}$, since $\eta^{\kappa}=\Lambda_\kappa^2\eta$ on $\Gamma$ and $v=\partial_t\eta-\fk$, we have
\begin{equation}
\begin{split}
&- \int_{\Gamma_2} \pa_3 q\bp^2\Delta_* \Lambda_\kappa^2\eta_j \a^\kappa_{j3} \a^{\kappa}_{i\ell}N_\ell\mathcal{V}_i
\\&\quad = \int_{\Gamma_2} (-\nabla q\cdot N )\a^\kappa_{j3}\bp^2\Delta_* \Lambda_\kappa^2\eta_j  \a^{\kappa}_{i3} \mathcal{V}_i
 \\&\quad=\int_{\Gamma_2}(-\nabla q\cdot N )\a^\kappa_{j3}\bp^2\Delta_* \Lambda_\kappa^2\eta_j  \a^{\kappa}_{i3} (\bar\partial^2\Delta_* v_i- \bp^2\Delta_* \Lambda_\kappa^2\eta \cdot\nak v_i)
  \\&\quad=I_1+I_2.
\end{split}
\end{equation}
Using Lemma \ref{IG_0}, we have
\begin{equation}\label{I_1}
\begin{split}
I_1&=\int_{\Gamma_2}(-\nabla q\cdot N )\a^\kappa_{j3}\bp^2\Delta_* \Lambda_\kappa^2\eta_j  \a^{\kappa}_{i3} \bar\partial^2\Delta_* v_i\\
&\lesssim \abs{\nabla q\cdot N\a^\kappa_{j3} \a^{\kappa}_{i3} }_{L^\infty}\abs{\bp^2\Delta_* \Lambda_\kappa^2\eta_j}_{H^{\frac{1}{2}}}\abs{\bar\partial^2\Delta_* v_i}_{H^{-\frac{1}{2}}}\\
&\lesssim\abs{\nabla q\cdot N\a^\kappa_{j3} \a^{\kappa}_{i3} }_{L^\infty}\abs{\eta}_{H^{4.5}}\abs{v}_{H^{3.5}}\\
&\lesssim P\left(\sup_{t\in[0,T]} \mathfrak{E}^{\kappa}(t)\right),
\end{split}
\end{equation}
\begin{equation}\label{I_2}
\begin{split}
I_2&=-\int_{\Gamma_2}(-\nabla q\cdot N )\a^\kappa_{j3}\bp^2\Delta_* \Lambda_\kappa^2\eta_j  \a^{\kappa}_{i3} \bp^2\Delta_* \Lambda_\kappa^2\eta \cdot\nak v_i\\
&\lesssim\abs{\nabla q\cdot N\a^\kappa_{j3}\a^{\kappa}_{i3}\nak v_i}_{L^\infty}\|\bp^2\Delta_* \Lambda_\kappa^2\eta_j \|_{L^2}\|\bp^2\Delta_* \Lambda_\kappa^2\eta\|_{L^2}\\
&\lesssim P\left(\sup_{t\in[0,T]} \mathfrak{E}^{\kappa}(t)\right).
\end{split}
\end{equation}

Finally, combining \eqref{gg1}, \eqref{gg2} and \eqref{gg3}, and using the estimates \eqref{ggg0}, \eqref{gggg3}, \eqref{commm1}, \eqref{ggiip},  \eqref{I_1} and  \eqref{I_2} we obtain
\begin{equation}\label{ohoh}
 \dfrac{d}{dt}\left(\int_{\Omega}\abs{\mathcal{V}}^2+\abs{\bar\partial^2\Delta_*( G_0^T\cdot\nabla\eta)}^2 + \int_{\Gamma_1}(-\nabla q\cdot N )\abs{\bp^2\Delta_* \Lambda_\kappa \eta_i \a^\kappa_{i3} }^2   \right)\leq  P\left(\sup_{t\in[0,T]} \mathfrak{E}^{\kappa}(t)\right).
\end{equation}
Integrating \eqref{ohoh} directly in time, by the a priori assumption \eqref{ini2}, we have
\begin{equation}\label{ohoh1}
\norm{\mathcal V(t)}_0^2+\norm{\bar\partial^4 (G_0^T\cdot\nabla \eta)(t)}_0^2+\abs{\bar\partial^2\Delta_*\Lambda_{\kappa}\eta_i \a^{\kappa}_{i3}(t)}^2_{L^2(\Gamma_1)}\leq M_0+TP\left(\sup_{t\in[0,T]} \mathfrak{E}^{\kappa}(t)\right).
\end{equation}
By the definition of $\mathcal{V}$, using \eqref{ohoh1} and  \eqref{etaest}, using the fundamental theorem of calculous, we get
\begin{equation}
\begin{split}
\norm{\bar\partial^4 v(t)}_0^2&\ls \norm{\bar\partial^2\Delta_* v(t)}_0^2\ls\norm{\mathcal V(t)}_0^2+\norm{\bar\partial^2\Delta_*\eta}_0^2\norm{\a^{\kappa}_{j\ell}\partial_{\ell} v}_{L^{\infty}}^2
\\&\leq M_0+TP\left(\sup_{t\in[0,T]} \mathfrak{E}^{\kappa}(t)\right).
\end{split}
\end{equation}
We thus conclude the proposition.
\end{proof}

Collecting the estimates above, we obtain the Proposition \ref{tane}.

\subsubsection{Curl and divergence estimates}\label{curle}
In view of the Hodge-type elliptic estimates, we can control one vector's all derivatives just by its curl, divergence and normal trace. Thus, we now derive the curl and divergence estimates.
We begin with the $\curl$ estimates.
\begin{proposition}
For $t\in [0,T]$ with $T\le T_\kappa$, it holds that
\begin{equation}
\label{curlest}
\norm{\curl v(t)}_{3}^2+\norm{\curl(G_0^T\cdot\nabla\eta)(t)}_{3}^2\leq M_0+TP\left(\sup_{t\in[0,T]} \mathfrak{E}^{\kappa}(t)\right).
\end{equation}
For a matrix $G$, we denote
\begin{equation}
(\curl G)_{ijk}=\partial_i G_{kj}- \partial_j G_{ki}.
\end{equation}
\end{proposition}
\begin{proof}
By taking the $\curl_{\a^{\kappa}}$ of the second equation in \eqref{approximate}, we have that
\begin{equation}
\curl_{\a^{\kappa}} \partial_tv=\curl_{\a^{\kappa}}\left((G_0^T \cdot\nabla)^2\eta\right),
\end{equation}
where $(\curl_{\ak}g)_i=\epsilon_{ij\ell}\a^{\kappa}_{jm}\partial_m g_{\ell}$.
It follows that
\begin{equation}\label{curl}
\partial_t(\curl_{\a^{\kappa}}v)_i-G_0^T\cdot\nabla\left(\curl_{\a^{\kappa}}\left(G_0^T\cdot\nabla \eta\right)\right)_i=\left(\left[\curl_{\a^{\kappa}}, G_0^T\cdot\nabla\right](G_0^T\cdot\nabla \eta)\right)_i+\epsilon_{ij\ell} \partial_t\a^{\kappa}_{jm} \partial_mv_{\ell}.
\end{equation}
Apply further $D^3$ to \eqref{curl} to get
\begin{equation}\label{curl2}
\partial_t(D^3\curl_{\a^{\kappa}}v)_i-G_0^T\cdot\nabla\left(D^3\curl_{\a^{\kappa}}\left(G_0^T\cdot\nabla \eta\right)\right)_i=F_i,
\end{equation}
with
\begin{equation}
F_i:=[D^3,G_0^T\cdot\nabla](\curl_{\a^{\kappa}}(G_0^T\cdot\nabla \eta))+D^3\left(\left(\left[\curl_{\a^{\kappa}}, G_0^T\cdot\nabla\right](G_0^T\cdot\nabla \eta)\right)_i+\epsilon_{ij\ell} \partial_t\a^{\kappa}_{jm} \partial_mv_{\ell}\right).
\end{equation}

Taking the $L^2$ inner product of \eqref{curl2} with $D^3\curl_{\a^{\kappa}}v$, by the integration by parts, we get
\begin{align}\label{j00}
&\dfrac{1}{2}\dfrac{d}{dt}\int_{\Omega}\abs{D^3\curl_{\a^{\kappa}}v}^2+\underbrace{\int_{\Omega}D^3\curl_{\a^{\kappa}}(G_0^T\cdot\nabla \eta)\cdot D^3\curl_{\a^{\kappa}}(G_0^T\cdot\nabla v)}_{\mathcal{J}_1}\\&\quad=\underbrace{\int_\Omega F \cdot D^3\curl_{\a^{\kappa}}v}_{\j_2}+\underbrace{\int_{\Omega}D^{3}\curl_{\a^{\kappa}}(G_0^T\cdot\nabla \eta)\cdot \left[D^3\curl_{\a^{\kappa}}, G_0^T\cdot\nabla\right] v}_{\j_3}.\nonumber
\end{align}
Since $v=\dt\eta-\fk$, we have
\begin{equation}
\label{j0}
\begin{split}
\j_1=&\dfrac{1}{2}\dfrac{d}{dt}\int_{\Omega}\abs{D^3\curl_{\a^{\kappa}}(G_0^T\cdot\nabla \eta)}^2
\underbrace{-\int_{\Omega}D^3\curl_{\a^{\kappa}}(G_0^T\cdot\nabla \eta)\cdot D^3\curl_{\a^{\kappa}}(G_0^T\cdot\nabla \fk)}_{\j_{1a}}\\&-\underbrace{\dfrac{1}{2}\int_{\Omega}D^3(\curl_{\a^{\kappa}}(G_0^T\cdot\nabla \eta))_i\cdot D^3(\epsilon_{ij\ell} \partial_t\a^{\kappa}_{jm} \partial_m(G_0^T\cdot\nabla\eta_\ell))}_{\j_{1b}}.
\end{split}
\end{equation}
By the estimates \eqref{fest22} and \eqref{tes1}, we obtain
\begin{equation}
\label{j0b}
\j_{1a}\ls \norm{G_0^T\cdot\nabla\eta}_4\norm{\a^{\kappa}}_3^2\norm{G_0^T\cdot\nabla\fk}_4\leq P\left(\sup_{t\in[0,T]} \mathfrak{E}^{\kappa}(t)\right).
\end{equation}
By the identity \eqref{partialF} and the estimates \eqref{tes0} and \eqref{fest1}, we have
\begin{equation}
\j_{1b}
\ls  \norm{\a^{\kappa}}_3\norm{D(G_0^T\cdot\nabla\eta)}_3\norm{\partial_t\a^{\kappa}}_3\norm{D(G_0^T\cdot\nabla\eta)}_3
  \le P\left(\sup_{t\in[0,T]} \mathfrak{E}^{\kappa}(t)\right).
\end{equation}
Hence, we obtain
\begin{equation}
\label{j1}
\j_1\ge \dfrac{1}{2}\dfrac{d}{dt}\int_{\Omega}\abs{D^3\curl_{\a^{\kappa}}(G_0^T\cdot\nabla \eta)}^2
-P\left(\sup_{t\in[0,T]} \mathfrak{E}^{\kappa}(t)\right).
\end{equation}

We now turn to estimate the right hand side of \eqref{j00}. By the estimates \eqref{tes1}--\eqref{tes0} and the identity \eqref{partialF}, we may have
\begin{equation}
\label{j2}
\begin{split}
\j_2&\leq  \norm{F}_0 \norm{\curl_{\ak} v}_3
\leq P\left(\norm{G_0}_4,  \norm{G_0^T\cdot\nabla\eta}_4, \norm{Dv}_3, \norm{G_0^T\cdot\nabla\a^{\kappa}}_3, \norm{\a^{\kappa}}_3\right)
\\&\leq P\left(\sup_{t\in[0,T]} \mathfrak{E}^{\kappa}(t)\right).
\end{split}
\end{equation}
Similarly,
\begin{align}\label{j3}
\nonumber\j_3&\ls \norm{D^3\curl_{\a^{\kappa}}(G_0^T\cdot\nabla \eta)}_0\norm{[D^3\curl_{\a^{\kappa}}, G_0^T\cdot\nabla] v}_0
\\&\leq P\left(\norm{G_0}_4,  \norm{G_0^T\cdot\nabla\eta}_4, \norm{Dv}_3, \norm{G_0^T\cdot\nabla\a^{\kappa}}_3, \norm{\a^{\kappa}}_3\right)
 \leq P\left(\sup_{t\in[0,T]} \mathfrak{E}^{\kappa}(t)\right).
\end{align}

Consequently, plugging the estimates \eqref{j1}--\eqref{j3} into \eqref{j00}, we obtain
\begin{equation} \label{kllj}
\dfrac{d}{dt}\int_{\Omega}\abs{D^3\curl_{\a^{\kappa}}v}^2+\abs{D^3\curl_{\a^{\kappa}}(G_0^T\cdot\nabla \eta)}^2
\le P\left(\sup_{t\in[0,T]} \mathfrak{E}^{\kappa}(t)\right).
\end{equation}
Integrating \eqref{kllj} directly in time, and applying the fundamental theorem of calculous,
\begin{equation}
\norm{\curl  f(t)}_3\le \norm{\curl_{\a^{\kappa}} f(t)}_3+\norm{\int_0^t\partial_t\a^{\kappa} d\tau D f(t)}_3,
\end{equation}
we then conclude the proposition.
\end{proof}

We now derive the divergence estimates.
\begin{proposition}
For $t\in [0,T]$ with $T\le T_\kappa$, it holds that
\begin{equation}
\label{divest}
\norm{\Div v(t)}_3^2+\norm{\Div(G_0^T\cdot\nabla\eta)(t)}_3^2 \leq TP\left(\sup_{t\in[0,T]} \mathfrak{E}^{\kappa}(t)\right).
\end{equation}
\end{proposition}
\begin{proof}
From $\divak v=0$, we see that
\begin{equation}
\Div v=-\int_0^t\partial_t\a^{\kappa}_{ij}\,d\tau\partial_jv_i.
\end{equation}
Hence, it is clear that by the identity \eqref{partialF} and the estimates \eqref{tes1} and \eqref{tes0},
\begin{equation}\label{divv1}
\norm{\Div v(t)}_{3}^2\le TP\left(\sup_{t\in[0,T]} \mathfrak{E}^{\kappa}(t)\right).
\end{equation}

From $\divak v=0$ again, we have
\begin{equation*}
\Div_{\a^{\kappa}}(G_0^T\cdot\nabla v)=[\Div_{\a^{\kappa}},G_0^T\cdot\nabla] v.
\end{equation*}
This together with the equation $v=\partial_t\eta-\fk$, we have
\begin{equation}
\label{divb}
\begin{split}
\partial_t\left(\Div_{\ak}\left(G_0^T\cdot \nabla \eta\right)\right)=\Div_{\a^{\kappa}}\left(G_0^T\cdot\nabla \fk\right)+[\Div_{\a^{\kappa}},G_0^T\cdot\nabla] v+ \partial_t \a^{\kappa}_{i\ell} \partial_{\ell}\left(G_0^T\cdot\nabla\eta_i\right).
\end{split}
\end{equation}
This implies that, by doing the $D^3$ energy estimate and using the estimates \eqref{tes1}--\eqref{fest22} and the identity \eqref{partialF},
\begin{equation}\label{di2}
\norm{\Div_{\ak}\left(G_0^T\cdot \nabla \eta\right)}_{3}^2
  \leq  TP\left(\sup_{t\in[0,T]} \mathfrak{E}^{\kappa}(t)\right).
\end{equation}
And then applying the fundamental theorem of calculous, and
\begin{equation}
\norm{\Div  f(t)}_3\le \norm{\Div_{\a^{\kappa}} f(t)}_3+\norm{\int_0^t\partial_t\a^{\kappa} d\tau D f(t)}_3,
\end{equation}
we arrive at
\begin{equation}\label{divv12}
\norm{\Div(G_0^T\cdot \nabla \eta)}_{3}^2
  \leq  TP\left(\sup_{t\in[0,T]} \mathfrak{E}^{\kappa}(t)\right).
\end{equation}

Consequently, we conclude the proposition by the estimates \eqref{divv1} and \eqref{divv12}.
\end{proof}

\subsubsection{Synthesis}

We now collect the estimates derived previously to conclude our estimates and also verify the a priori assumptions \eqref{ini2} and \eqref{inin3}. That is, we shall now present the
\begin{proof}[Proof of Theorem \ref{th43}]
It follows from the normal trace estimates \eqref{gga} that
\begin{equation}
\abs{\bar\partial^4 v\cdot N}_{-1/2}\ls \norm{\bar\partial^4 v}_0+\norm{\Div\bar\partial^3v}_0.
\end{equation}
Combining this and the estimates \eqref{00estimate}, \eqref{teee}, \eqref{curlest} and \eqref{divest},  by using the Hodge-type elliptic estimates \eqref{hodd} of Lemma \ref{hodge}, we obtain
\begin{align}
\norm{v}_4^2 \ls \norm{v}_0^2+\norm{\Div v}_{3}^2+\norm{\curl v}_{3}^2+\abs{\bar\partial v\cdot N}_{5/2}^2
\leq M_0+TP\left(\sup_{t\in[0,T]} \mathfrak{E}^{\kappa}(t)\right).
\end{align}
Similarly, we have
\begin{align}
\norm{G_0^T\cdot\nabla\eta}_4^2 \leq M_0+TP\left(\sup_{t\in[0,T]} \mathfrak{E}^{\kappa}(t)\right).
\end{align}
By these two estimates and \eqref{etaest}, \eqref{teee}, we finally get that
\begin{equation*}
\sup_{[0,T]}\mathfrak{E}^{\kappa}(t)\leq M_0+TP\left(\sup_{t\in[0,T]} \mathfrak{E}^{\kappa}(t)\right).
\end{equation*}
This provides us with a time of existence $T_1$ independent of $\kappa$ and an estimate on $[0,T_1]$ independent of $\kappa$ of the type:
\begin{equation}
\sup_{[0,T_1]}\mathfrak{E}^{\kappa}(t)\leq 2M_0.
\end{equation}

The proof of Theorem \ref{th43} is thus completed.
\end{proof}

\section{Proof of Theorem \ref{mainthm}}
\begin{proof}[Proof of Theorem \ref{mainthm}]
For each $\kappa>0$, we can construct the solutions to the $\kappa$-approximate system \eqref{approximate}  by a similar way in \cite[Section 5]{GuW_2016}. Briefly, we linearized the $\kappa$-approximate system and solve the linearized system by an aritifial viscosity method. Then a contract map method tells the existence of solutions to $\kappa$-approximate system \eqref{approximate}. Then the the existence of solutions to \eqref{eq:mhd} follows by taking $\kappa\rightarrow 0$.
We omit the details here.
\end{proof}

\section*{Conflict of Interest}
The authors declare that they have no conflict of interest.


\begin{thebibliography}{99}
\bibitem{AD}
T. Alazard, J. M. Delort.
Global solutions and asymptotic behavior for two dimensional gravity water waves.
\emph{Ann. Sci. \'Ec. Norm. Sup\'er.} \textbf{48}  (2015), no. 5, 1149--1238.

 \bibitem{Alinhac}
S. Alinhac.
Existence d'ondes de rar\'efaction pour des syst\`emes quasi-lin\'eaires hyperboliques multidimensionnels.(French. English summary) [Existence of
rarefaction waves for multidimensional hyperbolic quasilinear systems]
\emph{Comm. Partial Differential Equations} \textbf{14} (1989), no. 2, 173--230.


\bibitem{Chen_08}
G. Chen, Y. Wang.
Existence and stability of compressible current-vortex sheets in three-dimensional magnetohydrodynamics.
\emph{Arch. Ration. Mech. Anal.} \textbf{187} (2008), no. 3, 369--408.

\bibitem{Chen_17}
R. Chen, J. Hu, D. Wang.
Linear stability of compressible vortex sheets in two-dimensional elastodynamics.
\emph{Adv. Math.} \textbf{311}(2017), 18--60. 




\bibitem{CL_00}
D. Christodoulou, H. Lindblad.
On the motion of the free surface of a liquid.
\emph{Comm. Pure Appl. Math.} \textbf{53} (2000), no. 12, 1536--1602.


\bibitem{CMST}
J. F. Coulombel, A. Morando, P. Secchi, P. Trebeschi. A priori estimates for 3D incompressible current-vortex sheets. \emph{Comm. Math. Phys.} \textbf{311} (2012), no. 1, 247--275.

\bibitem{CS07}
D. Coutand, S. Shkoller.
Well-posedness of the free-surface incompressible Euler equations with or without surface tension.
\emph{J. Amer. Math. Soc.}  \textbf{20} (2007), no. 3, 829--930.

\bibitem{DS_10}
D. Coutand, S. Shkoller.
A simple proof of well-posedness for the free-surface incompressible Euler equations.
\emph{Discrete Contin. Dyn. Syst. Ser. S.} \textbf{3} (2010), no. 3, 429--449.
\bibitem{GMS1}
P. Germain, N. Masmoudi, J. Shatah.
Global solutions for the gravity water waves equation in dimension 3.
\emph{Ann. of Math. (2)} \textbf{175} (2012), no. 2, 691--754.
\bibitem{GMS2}
P. Germain, N. Masmoudi, J. Shatah.
Global solutions for capillary waves equation.
\emph{Comm. Pure Appl. Math.} \textbf{68} (2015), no. 4, 625--687.

\bibitem{GuW_2016}
X. Gu, Y. Wang.
On the construction of solutions to the free-surface incompressible ideal magnetohydrodynamic equations.
Preprint (2016), arXiv:1609.07013.

\bibitem{Gu_2017}
X. Gu.
Well-posedness of axially symmetric incompressible ideal magnetohydrodynamic equations with vacuum under the Rayleigh-Taylor sign condition
Preprint (2017), arXiv:1712.02152.

\bibitem{Hao_13}
C. Hao, T. Luo.
A priori estimates for free boundary problem of incompressible inviscid magnetohydrodynamic flows.
\emph{Arch. Ration. Mech. Anal.} \textbf{212} (2014), no.3, 805--847.

\bibitem{Hao_16}
C. Hao, D. Wang.
A priori estimates for the free boundary problem of incompressible neo-Hookean elastodynamics.
\emph{J. Differential Equations.} \textbf{261} (2016), no. 1, 712--737. 

\bibitem{IP}
A. Ionescu, F. Pusateri.
Global solutions for the gravity water waves system in 2D.
\emph{Invent. Math.} \textbf{199} (2015), no. 3, 653--804.

\bibitem{IP2}
A. Ionescu, F. Pusateri.
Global regularity for 2D water waves with surface tension.
\emph{Mem. Amer. Math. Soc.}, to appear.

\bibitem{Lannes}
D. Lannes.
Well-posedness of the water-waves equations.
\emph{J. Amer. Math. Soc.} \textbf{18} (2005), no. 3, 605--654.

\bibitem{LWZ_2018}
H. Li, W. Wang, Z. F. Zhang.
Well-posedness of the free boundary problem in incompressible elastodynamics.
Preprint (2018), arXiv:1802.08819.





\bibitem{Lindblad05}
H. Lindblad.
Well-posedness for the motion of an incompressible liquid with free surface boundary.
\emph{Ann. of Math. (2)} \textbf{162} (2005), no. 1, 109--194.


\bibitem{MasRou}
N. Masmoudi, F. Rousset.
Uniform regularity and vanishing viscosity limit for the free surface Navier-Stokes equations.
\emph{Arch. Ration. Mech. Anal.} (2016), DOI: 10.1007/s00205-016-1036-5.


\bibitem{Mo_14}
A. Morando, Y. Trakhinin, P. Trebeschi.
Well-posedness of the linearized plasma-vacuum interface problem in ideal incompressible MHD.
\emph{Quart. Appl. Math.} \textbf{72} (2014), no. 3, 549--587.


 \bibitem{N}
V. I. Nalimov.
The Cauchy-Poisson problem. (Russian)
\emph{Dinamika Splo$\breve{s}$n. Sredy Vyp. 18 Dinamika $\breve{Z}$idkost. so Svobod. Granicami.} \textbf{254} (1974), 104--210.




\bibitem{Secchi_13}
P. Secchi, Y. Trakhinin.
Well-posedness of the linearized plasma-vacuum interface problem.
\emph{Interfaces Free Bound.} \textbf{15} (2013), no. 3, 323--357.


\bibitem{Secchi_14}
P. Secchi, Y. Trakhinin.
Well-posedness of the plasma-vacuum interface problem.
\emph{Nonlinearity} \textbf{27} (2014), no. 3, 105--169.


\bibitem{SZ}
J. Shatah, C. Zeng.
Geometry and a priori estimates for free boundary problems of the Euler equation.
\emph{Comm. Pure Appl. Math.} \textbf{61} (2008), no. 5, 698--744.

\bibitem{Sun_15}
Y. Sun, W. Wang, Z. Zhang.
Nonlinear stability of current-vortex sheet to the incompressible MHD equations.
\emph{Comm. Pure Appl. Math.} \textbf{71} (2018), 356--403.

\bibitem{Sun_17}
Y. Sun, W. Wang, Z. Zhang. 
Well-posedness of the plasma-vacuum interface problem for ideal incompressible
MHD. Preprint (2017), arXiv:1705.00418.

\bibitem{Taylor}
\newblock M. Taylor,
\newblock Partial Differential Equations, Vol. I-III,
\newblock Berlin-Heidelberg-New York: Springer, (1996).


\bibitem{Trak_09}
Y. Trakhinin.
The existence of current-vortex sheets in ideal compressible magnetohydrodynamics.
\emph{Arch. Ration. Mech. Anal.} \textbf{191} (2009), no. 2, 245--310.


\bibitem{Trak_10}
Y. Trakhinin.
On the well-posedness of a linearized plasma-vacuum interface problem in ideal compressible MHD.
\emph{J. Differential Equations} \textbf{249} (2010), no. 10, 2577--2599

\bibitem{Trak_18}
Y. Trakhinin.
Well-posedness of the free boundary problem in compressible elastodynamics.
\emph{J. Differential Equations} \textbf{264} (2018), 1661--1715.


\bibitem{Wang_15}
Y. Wang, Z. Xin.
Vanishing viscosity and surface tension limits of incompressible viscous surface waves.
Preprint (2015), arXiv: 1504.00152.

\bibitem{Wu1}
S. Wu.
Well-posedness in Sobolev spaces of the full water wave problem in 2-D.
\emph{Invent. Math.} \textbf{130} (1997), no. 1, 39--72.


\bibitem{Wu2}
S. Wu.
Well-posedness in Sobolev spaces of the full water wave problem in 3-D.
\emph{J. Amer. Math. Soc.} \textbf{12} (1999), no. 2, 445--495.


\bibitem{Wu3}
S. Wu.
Almost global wellposedness of the 2-D full water wave problem.
\emph{Invent. Math.} \textbf{177} (2009), no. 1, 45--135.


\bibitem{Wu4}
S. Wu.
Global wellposedness of the 3-D full water wave problem.
\emph{Invent. Math.} \textbf{184} (2011), no. 1, 125--220.



\bibitem{ZZ}
P. Zhang, Z. Zhang.
On the free boundary problem of three-dimensional incompressible Euler equations.
\emph{Comm. Pure Appl. Math.} \textbf{61} (2008), no. 7, 877--940.

\end{thebibliography}
\end{document}